\documentclass[10pt, reqno]{amsart}
\usepackage[margin=1in]{geometry}
\usepackage{enumerate,amsmath,amsthm,wasysym,bbold,tensor,combelow,epigraph}
\usepackage{amssymb}
\usepackage{amsfonts}
\usepackage{mathscinet}
\usepackage[dvipsnames]{xcolor}
\usepackage{tikz-cd}
\usepackage{tikz,keyval,pgfsys,graphicx}

\newtheorem{thm}{\normalfont\scshape Theorem}[section]

\newtheorem*{lem*}{\normalfont\scshape Lemma}
\newtheorem{prop}[thm]{\normalfont\scshape Proposition}
\newtheorem{lem}[thm]{\normalfont\scshape Lemma}

\newtheorem{conj}[thm]{\normalfont\scshape Conjecture}

\theoremstyle{definition}
\newtheorem{defn}[thm]{\normalfont\scshape Definition}

\theoremstyle{remark}
\newtheorem{rem}[thm]{Remark}
\theoremstyle{remark}

\allowdisplaybreaks

\begin{document}

\binoppenalty=10000
\relpenalty=10000

\numberwithin{equation}{section}

\newcommand{\sss}{\mathfrak{s}}
\newcommand{\rrr}{\mathfrak{r}}
\newcommand{\ccc}{\mathfrak{c}}
\newcommand{\nnn}{\mathfrak{n}}
\newcommand{\ch}{\mathrm{ch}}
\newcommand{\qqq}{\mathfrak{q}}
\newcommand{\ddd}{\mathfrak{d}}
\newcommand{\fff}{\mathfrak{f}}
\newcommand{\QQ}{\mathbb{Q}}
\newcommand{\ZZ}{\mathbb{Z}}
\newcommand{\Hilb}{\mathrm{Hilb}}
\newcommand{\CC}{\mathbb{C}}
\newcommand{\PP}{\mathcal{P}}
\newcommand{\Hom}{\mathrm{Hom}}
\newcommand{\Rep}{\mathrm{Rep}}
\newcommand{\MM}{\mathfrak{M}}
\newcommand{\gl}{\mathfrak{gl}}
\newcommand{\ppp}{\mathfrak{p}}
\newcommand{\VV}{\mathcal{V}}
\newcommand{\NN}{\mathbb{N}}
\newcommand{\OO}{\mathcal{O}}
\newcommand{\HH}{\mathcal{H\kern-.44em H}}
\newcommand{\git}{/\kern-.35em/}
\newcommand{\KK}{\mathbb{K}}
\newcommand{\Proj}{\mathrm{Proj}}
\newcommand{\quot}{\mathrm{quot}}
\newcommand{\core}{\mathrm{core}}
\newcommand{\Sym}{\mathrm{Sym}}
\newcommand{\FF}{\mathbb{F}}
\newcommand{\Span}[1]{\mathrm{span}\{#1\}}
\newcommand{\UTor}{U_{\qqq,\ddd}(\ddot{\mathfrak{sl}}_\ell)}
\newcommand{\Uaff}{U_{\qqq,\ddd}(\dot{\mathfrak{sl}}_\ell)}
\newcommand{\Uaffh}{U_{\qqq,\ddd}^h(\dot{\mathfrak{sl}}_\ell)}
\newcommand{\Uaffv}{U_{\qqq,\ddd}^v(\dot{\mathfrak{sl}}_\ell)}
\newcommand{\UAff}{U_{\qqq,\ddd}(\dot{\mathfrak{gl}}_\ell)}
\newcommand{\UAffh}{U_{\qqq,\ddd}^h(\dot{\mathfrak{gl}}_\ell)}
\newcommand{\UAffv}{U_{\qqq,\ddd}^v(\dot{\mathfrak{gl}}_\ell)}
\newcommand{\Heis}{\mathcal{H}}
\newcommand{\Heish}{\mathcal{H}_h}
\newcommand{\Heisv}{\mathcal{H}_v}
\newcommand{\Sss}{\mathcal{S}}

\title{Shuffle approach to wreath Pieri operators}
\author{Joshua Jeishing Wen}
\address{Department of Mathematics, Northeastern University, Boston, MA, USA}
\email{j.wen@northeastern.edu}
\maketitle

\begin{abstract}
We describe a way to study and compute Pieri rules for wreath Macdonald polynomials using the quantum toroidal algebra.
The Macdonald pairing can be naturally generalized to the wreath setting, but the wreath Macdonald polynomials are no longer collinear with their duals.
We establish the relationship between these dual polynomials and the quantum toroidal algebra, and we outline a way to compute norm formulas.
None of the aforementioned formulas are successfully computed in this paper.
\end{abstract}

\section{Introduction}
In symmetric function theory, one finds tremendous utility in having precise control over creation and annihilation.
Specifically, by creation we mean some form of \textit{Pieri rules}.
The point of this paper is to establish a way to study Pieri rules for \textit{wreath Macdonald polynomials}.
To motivate this, let us first consider the simplest case.
Let $\Lambda$ be the ring of symmetric functions in infinitely many variables, $e_n$ be the $n$th elemetary symmetric function, $h_n$ be the $n$th complete symmetric function, and $s_\lambda$ be the Schur function indexed by the partition $\lambda$.
The Pieri rules take the form (cf. \cite{Mac}):
\begin{align}
\label{HSPieri}
h_ns_\mu&= \sum_{\lambda}s_\lambda\\
e_ns_\mu&= \sum_{\lambda}s_\lambda
\label{ESPieri}
\end{align}
In both formulas, the summation is over $\lambda$ containing $\mu$ such that $|\lambda\backslash\mu|=n$.
For (\ref{HSPieri}), we require $\lambda\backslash\mu$ to contain no \textit{vertically} adjacent boxes, whereas for (\ref{ESPieri}), we require $\lambda\backslash\mu$ to contain no \textit{horizontally} adjacent boxes.

For annihilation, we consider the \textit{Hall pairing} $\langle-, -\rangle$ on $\Lambda$ under which $\left\{ s_\lambda \right\}$ form an orthonormal basis and take the adjoint to the Pieri rules.
Although it is not usually presented in this fashion, one can use (\ref{HSPieri}) and the fact that the complete symmetric function $h_\lambda$ is dual to the \textit{monomial symmetric function} $m_\lambda$ under $\langle - , -\rangle$ to obtain the \textit{tableaux sum formula} for $s_\lambda$.
Namely, after specializing to $n$ variables, we have
\[
s_\lambda(x_1,\ldots, x_n)=\sum_{T}x^T
\]
Here, the summation is over column-strict Young tableaux of shape $\lambda$.
Such a tableau $T$ can be viewed as a nested sequence of subpartitions:
\[
\varnothing=\lambda^{(0)}\subset\lambda^{(1)}\subset\cdots\subset\lambda^{(n)}=\lambda
\]
We set $x^T$ as the monomial where the exponent of $x_i$ is $|\lambda^{(i)}\backslash\lambda^{(i-1)}|$.
This is but one facet of the rich theory of skew Schur functions.

\subsection{The Macdonald case}
The \textit{Macdonald polynomials} $\left\{ P_\lambda \right\}$ are a basis of the twice-deformed ring of symmetric functions $\Lambda_{q,t}:=\Lambda\otimes \CC(q,t)$.
They can be obtained by deforming $\langle-,-\rangle$ into the \textit{Macdonald pairing} $\langle-,-\rangle_{q,t}$ and performing the Gram-Schmidt process on $\left\{ s_\lambda \right\}$ with respect to the new pairing and the dominance order on partitions.
$P_\lambda$ is normalized so that its leading term $s_\lambda$ has coefficient $1$, and thus its norm is not necessarily 1---its dual is denoted $Q_\lambda$.
Here, the Pieri rules involve $e_n$ and the \textit{transformed} complete symmetric functions $g_n$:
\begin{align}
\label{GPPieri}
g_nQ_\mu&= \sum_{\lambda}\prod_{ i<j\le\ell(\lambda)}\frac{\left( q^{\mu_i-\lambda_j+ 1}t^{j-i-1};q \right)_{\lambda_i-\mu_i}}{\left( q^{\mu_i-\lambda_j}t^{j-i};q \right)_{\lambda_i-\mu_i}}
\prod_{ i\le j\le\ell(\mu)}\frac{\left( q^{\mu_i-\mu_j}t^{j-i +1};q \right)_{\lambda_i-\mu_i}}{\left( q^{\mu_i-\mu_j+1}t^{j-i};q \right)_{\lambda_i-\mu_i}}Q_\lambda\\
e_nP_\mu&= \sum_{\lambda}\prod_{\substack{i<j\\ \lambda_i=\mu_i \\\lambda_j=\mu_j +1}}\frac{\left(1-q^{\mu_i-\mu_j}t^{j -i - 1}\right)\left( 1-q^{\lambda_i-\lambda_j}t^{j-i+1} \right)}{\left( 1-q^{\mu_i-\mu_j}t^{j - i} \right)\left( 1-q^{\lambda_i -\lambda_j}t^{j-i} \right)}P_\lambda
\label{EPPieri}
\end{align}
Here, the partitions appearing in (\ref{GPPieri}) are as in (\ref{HSPieri}) and those appearing in $(\ref{EPPieri})$ are as in (\ref{ESPieri}).
The factor $(x;q)_k$ appearing in (\ref{GPPieri}) denotes the $q$-Pochhammer symbol:
\[
(x;q)_k:=\prod_{i=0}^{k-1}\left( 1-q^ix \right)
\]

The basis $\left\{ g_\lambda \right\}$ built out of $\left\{ g_n \right\}$ is dual under $\langle-,-\rangle_{q,t}$ to $\left\{m_\lambda \right\}$, and thus (\ref{GPPieri}) yields the tableaux sum formula for $P_\lambda$ specialized to $n$ variables:
\[
P_\lambda(x_1,\ldots, x_n)=\sum_{T}\prod_{k=1}^n
\psi_{\lambda^{(k)}/\lambda^{(k-1)}}x^T
\]
where $\psi_{\lambda/\mu}$ is the coefficient from (\ref{GPPieri}).
On the other hand, Macdonald \cite{Mac} gives a clever proof of the $e_n$ Pieri rule (\ref{EPPieri}) that simultaneously establishes three other fundamental formulas for Macdonald polynomials:
\begin{itemize}
\item \textit{Principal evaluation}: For $\lambda$ with $\ell(\lambda)\le n$, 
\[
P_\lambda(t^{n-1},t^{n-2}, \cdots, 1)=t^{n(\lambda)}\prod_{\square\in\lambda}\frac{1-q^{a'(\square)}t^{n-\ell'(\square)}}{1-q^{a(\square)}t^{\ell(\square)+1}}
\]
Here, $n(\lambda)=\sum_{i}(i-1)\lambda_i$, and for $\square=(a,b)\in\lambda$ (cf. \ref{YoungMaya} for notation), $a'(\square)=a$, $\ell'(\square)=b$, $a(\square)=\lambda_b-a$, and $\ell(\square)=\tensor[^t]{\lambda}{}_a-b$, where $\tensor[^t]{\lambda}{}$ is the transposed partition.
\item \textit{Norm formula}: We have:
\begin{equation}
\langle P_\lambda,P_\lambda\rangle_{q,t}=\prod_{\square\in\lambda}\frac{1-q^{a(\square)+1}t^{\ell(\square)}}{1-q^{a(\square)}t^{\ell(\square)+1}}
\label{NormForm}
\end{equation}
\item \textit{Evaluation duality}: For $\lambda$ and $\mu$ with $\ell(\lambda),\ell(\mu)\le n$,
\[
\frac{P_\lambda(q^{\mu_1}t^{n-1}, q^{\mu_2}t^{n-2},\ldots, q^{\mu_n})}{P_\lambda(t^{n-1},t^{n-2},\ldots 1)}
=\frac{P_\mu(q^{\lambda_1}t^{n-1}, q^{\lambda_2}t^{n-2},\ldots, q^{\lambda_n})}{P_\mu(t^{n-1},t^{n-2},\ldots 1)}
\]
\end{itemize}
Thus, in this case, the Pieri rules play a crucial role in establishing these widely-used aspects of Macdonald polynomials.

\subsection{Quantum toroidal algebras}
With a view towards generalization, we recall the action of the rank 1 \textit{quantum toroidal algebra} $U_{q,t}(\ddot{\mathfrak{gl}}_1)$ on $\Lambda_{q,t}$.
This action is was constructed by Feigin \textit{et. al.} \cite{FHHSY} via \textit{vertex operators} that were already present in prior work of Jing \cite{JingMac} and Garsia-Haiman-Tesler \cite{GHT}.
Haiman's proof of Macdonald positivity (cf. \cite{Haiman}) showed that it was natural to assign Macdonald polynomials to fixed point classes in torus-equivariant K-theory of the Hilbert scheme of points on $\CC^2$, and this yielded a different perspective on the action of $U_{q,t}(\ddot{\mathfrak{gl}}_1)$ on Macdonald polynomials.
Feigin-Tsymbaliuk \cite{FeiTsymK} and Schiffmann-Vasserot \cite{SchiffVass} reconstructed this action using correspondences, and from this perspective, the generators of $U_{q,t}(\ddot{\mathfrak{gl}}_1)$ act by very fine creation and annihilation operators.

At first glance, these creation and annihilation operators seem to have little to do with symmetric functions---they are just operators on a vector space indexed by partitions.
However, Feigin-Tsymbaliuk were able to recover the Pieri rules through a very roundabout manner.
Through work of Negu\cb{t} \cite{NegutShuff}, $U_{q,t}(\ddot{\mathfrak{gl}}_1)$ is isomorphic to a \textit{shuffle algebra}, a space of functions endowed with an exotic product structure---by abuse of notation, we will denote the function variables using $\{x_i\}$.
Negu\cb{t}'s isomorphism provides a completely different way to think about $U_{q,t}(\ddot{\mathfrak{gl}}_1)$, and elements that are easy to express on one side of the isomorphism can be very complicated to write down in the other side.
From \cite{FHHSY}, one can expect that a scalar multiple of the shuffle element
\begin{equation}
K_n:=\prod_{1\le i<j\le n}\frac{\left( x_i-qx_j \right)\left( x_j-qx_i \right)}{(x_i-x_j)^2}
\label{KnShuff}
\end{equation}
acts as multlpication by $e_n$, and in \cite{FeiTsymK}, the authors show that the matrix elements of its action on Macdonald polynomials agree with the Pieri rule (\ref{EPPieri}).
Following the ideas outlined in Section \ref{Normal} below, one can also re-derive the norm formula (\ref{NormForm}).
Using a toroidal approach to the Cherednik-Macdonald-Mehta identity (cf. \cite{CNO}), we expect a toroidal derivation of principal evaluation and evaluation duality, but that is outside the scope of this paper.
Overall, we wish to convey that one can derive the basic elements of Macdonald theory using $U_{q,t}(\ddot{\mathfrak{gl}}_1)$, although this path lacks the elegance of Macdonald's simultaneous proof.

\subsection{Wreath Macdonald polynomials}
Defined by Haiman, the wreath Macdonald polynomials are a generalization of the Macdonald polynomials from the symmetric groups $\Sigma_n$ to their wreath products with a fixed cyclic group $\ZZ/\ell\ZZ$.
Using a wreath analogue of the Frobenius characteristic, we can view them as elements of $\Lambda_{q,t}^{\otimes\ell}$, whose bases are naturally indexed by $\ell$-tuples of partitions.
A basic component of their definition is the decomposition of an ordinary partition into its $\ell$-core and $\ell$-quotient, which we review in \ref{CoreQuot} below.
Roughly speaking, the $\ell$-quotient is an $\ell$-tuple of partitions that records the ribbons of length $\ell$ we can peel of an ordinary partition and the $\ell$-core is the partition that remains when all ribbons have been peeled off.
The wreath Macdonald polynomials are again indexed by ordinary partitions, and we will recycle notation and denote them by $\left\{ P_\lambda \right\}$.
Letting $\lambda$ range over partitions with the same $\ell$-core, we obtain a basis of $\Lambda_{q,t}^{\otimes\ell}$.
We have elementary symmetric functions $e_n(p)$ for each tensorand in $\Lambda_{q,t}^{\otimes\ell}$, and thus the $e_n$ Pieri rules come in $\ell$ \textit{colors}.

Wreath analogues of the standard elements of Macdonald theory are still largely unexplored, and we suspect part of the reason is that things are just more complicated in this setting.
A key component of Macdonald's simultaeneous proof that we have yet to mention are the \textit{Macdonald operators}, a family of difference operators diagonalized by the Macdonald polynomials.
Wreath analogues of Macdonald operators were discovered in our work with Daniel Orr and Mark Shimozono \cite{OSW}, and we invite the reader to compare the wreath Macdonald operators with their predecessors to get a sense of the difference in complexity.
Even if these operators were simple, the inductive argument used by Macdonald fails on the spot.
The main problem here is that one has little control over the partitions $\lambda$ that appear in the expression
\[
e_n(p)P_\mu=\sum_\lambda c_{\lambda,\mu}(q,t)P_\lambda,
\]
especially with respect to dominance order.

Something that does generalize is the quantum toroidal algebra.
In our previous work \cite{WreathEigen}, we laid the groundwork for studying wreath Macdonald polynomials using an action of the higher rank quantum toroidal algebra $\UTor$.
These higher rank algebras are also isomorphic to certain shuffle algebras \cite{NegutTor}, and we found shuffle elements $E_{p,n}$ generalizing $K_n$ (\ref{KnShuff}) whose action corresponds to multiplication by $e_n(p)$.
Finding wreath Pieri rules then amounts to computing matrix elements for $E_{p,n}$ with respect to the wreath Macdonald basis.
This paper lays the necessary preparation for such a calculation.
We also discuss a natural wreath generalization of the Macdonald pairing and give a pathway towards computing norm formulas.
An interesting feature here is that $P_\lambda$ is no longer collinear to its dual and the two are related by the square of the \textit{Miki automorphism}.
Despite all this preparation, we were unable to carry out these calculations---the formulas are too complicated and the combinatorics too chaotic.
Thus, this paper is the site of an abandoned expedition.

\subsection{Outline}
Section 1 goes over the necessary elements for defining and working with wreath Macdonald polynomials: wreath products, the core-quotient decomposition, etc.
We also introduce the wreath Macdonald pairing and dual wreath Macdonald polynomials.
Section 2 covers the quantum toroidal algebra and its realization as a shuffle algebra.
The key points are formulas for shuffle elements corresponding to $e_n(p)$ and the appropriate generalization of $g_n$ as well as a way to compute matrix elements of shuffle elements in the \textit{Fock representation}.
Finally, Section 3 isolates a key computation that yields the norm formula and is a prerequisite to computing wreath Pieri rules using $\UTor$.

\subsection{Acknowledgements}
I would like to thank Daniel Orr and Mark Shimozono for their collaboration---this paper would not see the light of day if not for their interest.
This work received support from NSF-RTG grant ``Algebraic Geometry and Representation Theory at Northeastern University'' (DMS-1645877).

\section{Wreath Macdonald polynomials}
Here, we introduce the basics of wreath Macdonald polynomials.
We recommend the excellent article of Orr and Shimozono \cite{OSWreath} for a deeper and broader introduction.

\subsection{Wreath products}
Let
\[\Gamma_n:=(\ZZ/\ell\ZZ)^n\rtimes\Sigma_n\]
denote the wreath product of $\ZZ/\ell\ZZ$ and the symmetric group $\Sigma_n$.
We denote by $(\ZZ/\ell\ZZ)_*$ and $(\ZZ/\ell\ZZ)^*$ the set of conjugacy classes and characters, respectively, of $\ZZ/\ell\ZZ$.
To distinguish between the two, we will write elements of $(\ZZ/\ell\ZZ)_*$ as powers of a generator $\epsilon$, whereas we will index $(\ZZ/\ell\ZZ)^*$ additively by simply using integers modulo $\ell$.
Specifically, for $i\in\ZZ/\ell\ZZ$, let $\gamma_i\in(\ZZ/\ell\ZZ)^*$ be the character
\[\gamma_i(\epsilon^j)=\zeta^{ij}\]
where $\zeta=e^{\frac{2\pi i}{\ell}}$.
Finally, let 
\[\Lambda\cong\CC[p_n]_{n=1}^\infty\]
denote the ring of symmetric functions.
We will apply the general framework of Chapter I, Appendix B of \cite{Mac} to study representations of $\Gamma_n$.
All results reviewed in this subsection appear in \textit{loc. cit.}.

\subsubsection{Symmetric functions}
Consider the ring
\[\Lambda^{\otimes\ell}\cong\CC[p_n(c)]_{c\in (\ZZ/\ell\ZZ)_*}\]
For a partition $\lambda$, we set
\[p_\lambda(\epsilon^i):=\prod_kp_{\lambda_k}(\epsilon^i)\]
We will also need the set of ring generators indexed instead by $(\ZZ/\ell\ZZ)^*$ and given by
\[p_n(i):=\sum_{j=0}^{\ell-1}\zeta^{ij}p_n(\epsilon^j)\]
For each $i\in\ZZ/\ell\ZZ$, we have an isomorphism $\Lambda\cong\CC[p_n(i)]$ given by $p_n\mapsto p_n(i)$.
We define $e_n(i)$ and $h_n(i)$ to be the images of $e_n$ and $h_n$, respectively, under this isomorphism.
Similarly, we define $p_\lambda(i)$, $e_\lambda(i)$, $h_\lambda(i)$, and $s_\lambda(i)$ as the images of $p_\lambda$, $e_\lambda$, $h_\lambda$, and $s_\lambda$, respectively.
We will follow the conventions of \cite{WreathEigen}, so we have
\begin{equation}
e_\lambda=\prod_{k=1}^{\ell(\lambda)} e_{\tensor[^t]{\lambda}{}_k}
\label{ElemConv}
\end{equation}
where $\ell(\lambda)$ is the length.
For an $\ell$-multipartition $\vec{\lambda}=(\lambda^0,\ldots,\lambda^{\ell-1})$, we define
\[p_{\vec{\lambda}}:=\prod_{i\in\ZZ/\ell\ZZ}p_{\lambda^i}(i)\]
We similarly define $e_{\vec{\lambda}}$, $h_{\vec{\lambda}}$, and $s_{\vec{\lambda}}$.
Finally, we set
\[p_{\vec{\lambda}}^c:=\prod_{i\in\ZZ/\ell\ZZ}p_{\lambda^i}(\epsilon^i)\]

\subsubsection{Wreath Hall pairing}
The \textit{Hall pairing} on $\Lambda^{\otimes\ell}$ is given by
\[\langle p_{\vec{\lambda}},p_{\vec{\mu}}\rangle=\delta_{\vec{\lambda},\vec{\mu}}z_{\vec{\lambda}}\]
where if we write $\lambda^i=(1^{m_1^i}2^{m_2^i}\ldots)$, then
\[z_{\vec{\lambda}}=\prod_{i\in\ZZ/\ell\ZZ}\prod_k k^{m_k^i}m_{k}^i!\]
Under this pairing, $\{s_{\vec{\lambda}}\}$ forms an orthonormal basis.

\begin{rem}
In Chapter I, Appendix B of \cite{Mac}, the Hall pairing is written in terms of the conjugacy class basis $\{p_{\vec{\lambda}}^c\}$.
However, the same functions $\{s_{\vec{\lambda}}\}$ still form an orthogonal basis under the pairing in \textit{loc. cit.}, from which it follows that we can define the Hall pairing as we do here.
\end{rem}

\subsubsection{Wreath Frobenius characteristic}\label{FrobChar}
Let
\[R_n:=\Rep(\Gamma_n)\]
In analogy with the study of the representation theory of symmetric groups, it is useful to consider the representations of all $\Gamma_n$:
\[R:=\bigoplus_n R_n\]
We can endow $R$ with the structure of a $\CC$-algebra under the \textit{induction product}: if $[V]\in\Gamma_n$ and $[W]\in\Gamma_m$,
\[[V]*[W]:=\left[\mathrm{Ind}_{\Gamma_n\times\Gamma_m}^{\Gamma_{m+n}}(V\boxtimes W)\right]\in R_{n+m}\]
Observe as well that each $R_n$ has an inner product given by the \textit{Hom pairing}:
\[
\langle [V],[W]\rangle=\dim\Hom_{\Gamma_n}(V,W)
\]
for actual (i.e. nonvirtual) representations $V,W$ of $\Gamma_n$.

The irreducible representations of $\Gamma_n$ are indexed by $\ell$-multipartitions $\vec{\lambda}$ with $|\vec{\lambda}|=n$.
For such $\vec{\lambda}$, we denote the corresponding irrep by $V_{\vec{\lambda}}$.
The classes $[V_{\vec{\lambda}}]$ form an orthonormal basis under the Hom pairing.
We have the following result:
\begin{prop}
The linear map $R\mapsto\Lambda^{\otimes\ell}$ induced by
\[[V_{\vec{\lambda}}]\mapsto s_{\vec{\lambda}}\]
is both an isometry and a ring isomorphism.
\end{prop}
\noindent We call this map the \textit{(wreath) Frobenius characteristic}.

\subsubsection{Matrix plethysm\protect\footnote{We are indebted to D. Orr and M. Shimozono for this notion.}}
It will be useful to introduce plethystic notation for $\Lambda^{\otimes \ell}$.
However, we will have $\ell$ alphabets $\left\{ X^{(i)} \right\}$ indexed by $i\in\ZZ/\ell\ZZ$:
\[
\textstyle p_n(i)=p_n\left[X^{(i)}\right]
\]
The index $i$ will be called the \textit{color} of the alphabet.
We will use $X^{\bullet}$ to denote the alphabets over all colors and reserve $X^{(i)}$ for the specific color $i$.
For example,
\[
s_{\vec{\lambda}}=s_{\vec{\lambda}}[X^\bullet]
\]
but $s_{\vec{\lambda}}\left[ X^{(i)} \right]$ would only make sense if the only nonempty component of $\vec{\lambda}$ was its $i$th coordinate. 

Aside from the usual plethystic automorphisms of $\Lambda^{\otimes\ell}$ given by scaling each alphabet, we will also have automorphisms that mix colors in a linear fashion.
We will only consider such transformations built out of two basic operators.
First, we will use the \textit{upward color shift} $\sigma$:
\[
\sigma X^{(i)}=X^{(i+1)}
\]
for all $i\in\ZZ/\ell\ZZ$.
For an example of how we will use this operator, consider for an indeterminate $s$ the transformations
\[
\textstyle 
p_n\left[ X^{(i)} \right]\mapsto p_n\left[(1-s\sigma^{\pm 1})X^{(i)} \right]=p_n\left[ X^{(i)} \right]-s^np_n\left[ X^{(i\pm 1)} \right]
\]
One can directly compute that the inverse of the transformations above are given by
\[
{\textstyle p_n\left[X^{(i)}  \right]}\mapsto \frac{\sum_{j=0}^{\ell-1}s^{nj}p_n\left[X^{(i\pm j)}\right]}{1-s^{n\ell}}=:{\textstyle p_n\left[(1-s\sigma^{\pm 1})^{-1} X^{(i)} \right]}
\]
The second basic operator is the \textit{color negation} $\iota$:
\[
\iota X^{(i)}=X^{(-i)}
\]
This corresponds to taking the dual representation in $R$.
We close by noting that $\sigma$ and $\iota$ satisfy the following commutation relation:
\[
\iota\sigma=\sigma^{-1}\iota
\]
%

\subsection{Cores and quotients}\label{CoreQuot}
We will now review the core-quotient decomposition of a partition.

\subsubsection{Young-Maya correspondence}\label{YoungMaya}
In this paper, we will utilize two visual representations of a partition: its \textit{Young diagram} and its \textit{Maya diagram}.
We will follow the French convention for Young diagrams.
For example, below is the Young diagram for $(5,4,1)$:

\vspace{.1in}

\centerline{\begin{tikzpicture}[scale=.5]
\draw (0,0)--(5,0)--(5,1)--(4,1)--(4,2)--(1,2)--(1,3)--(0,3)--(0,0);;
\draw (0,1)--(4,1);;
\draw (0,2)--(1,2);;
\draw (1,0)--(1,2);;
\draw (2,0)--(2,2);;
\draw (3,0)--(3,2);;
\draw (4,0)--(4,1);;
\end{tikzpicture}}

\vspace{.1in}

\noindent We call the boxes in the diagram \textit{nodes} and index them according to their upper right coordinate, wherein each square has dimensions $1\times 1$ and the bottom left corner of the partition has coordinate $(0,0)$.
Thus, the topmost node in $(5,4,1)$ has coordinate $(1,3)$.
For a node $\square=(i,j)$, we define its \textit{content} by $c(\square):=j-i$.
Thus, the content is constant along integer translates of the line $y=x$; we call these translates \textit{content lines}.
For $i\in\ZZ/\ell\ZZ$, a node $\square$ is called an $i$-node if $c(\square)\equiv i\hbox{ mod }\ell$.
Finally, for a positive integer $k$, we define a \textit{$k$-strip} in $\lambda$ to be a contiguous set of $k$ nodes on the outer rim of $\lambda$ whose removal leaves behind another partition.
We call its node with the greatest content (northwesternmost) its \textit{initial node} and its node with the least content (southeasternmost) its \textit{final node}.

A \textit{Maya diagram} is a function $m:\ZZ\rightarrow\left\{ \pm1 \right\}$ such that $m(n)=-1$ for all $n\gg0$ and $m(n)=1$ for all $n\ll0$.
We can visualize this as a string of beads indexed by the integers, wherein the bead assigned to $n$ is black if $m(n)=1$ and white if $m(n)=-1$.
In our conventions, we will always draw our string with index increasing towards the \textit{left}.
Moreover, we will draw a notch between the indices 0 and -1; we call this the \textit{central line}.
The Maya diagram where all beads right of the central line are black and all beads left of the central line are white is called the \textit{vacuum diagram}.
We give an example below:

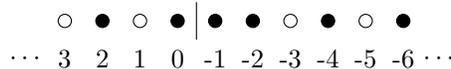
\vspace{.1in}

\centerline{\begin{tikzpicture}[scale=.5]
\draw (-4, 0) node {$\cdots$};;
\draw (-3,0) node {3};;
\draw (-3,1) circle (5pt);;
\draw (-2,0) node {2};;
\draw[fill=black] (-2,1) circle (5pt);;
\draw (-1,0) node {1};;
\draw (-1,1) circle (5pt);;
\draw (0,0) node {0};;
\draw[fill=black] (0,1) circle (5pt);;
\draw (.5,.5)--(.5,1.5);;
\draw (1,0) node {-1};;
\draw[fill=black] (1,1) circle (5pt);;
\draw (2,0) node {-2};;
\draw[fill=black] (2,1) circle (5pt);;
\draw (3,0) node {-3};;
\draw (3,1) circle (5pt);;
\draw (4,0) node {-4};;
\draw[fill=black] (4,1) circle (5pt);;
\draw (5,0) node {-5};;
\draw (5,1) circle (5pt);;
\draw (6,0) node {-6};;
\draw[fill=black] (6,1) circle (5pt);;
\draw (7,0) node {$\cdots$};;
\end{tikzpicture}}

\vspace{.1in}

\noindent For a Maya diagram $m$, we would expect beads to the left of the central line to be white and beads to the right of the central line to be black.
The \textit{charge} $c(m)$ measures the discrepancy between the number of white beads on the right and the number of black beads on the left:
\[c(m)=|\left\{ n<0 :m(n)=-1 \right\}|-|\left\{ n\ge 0 : m(n)=1 \right\}|\]

The \textit{Young-Maya correspondence} assigns to a partition $\lambda$ a Maya diagram $m(\lambda)$.
It is given as follows.
First, tilt the Young diagram of $\lambda$ counter-clockwise by 45 degrees so that it is in the Russian convention.
Draw the content lines and index the gap left of the content $n$ line by $n$.
Within each of these gaps, there is a segment of the outer border of the Young diagram.
We set $m(\lambda)(n)=1$ if the segment in the gap for $n$ has slope 1 and $m(\lambda)(n)=-1$ if it has slope -1.
For the indices $n$ less than the least content of a node in $\lambda$ minus 1, we set $m(\lambda)(n)=1$, and likewise for indices greater than the greatest content of a node in $\lambda$, we set $m(\lambda)(n)=-1$.
In Figure \ref{fig:YoungMayaEx}, we give the example of $\lambda=(5,4,1)$.

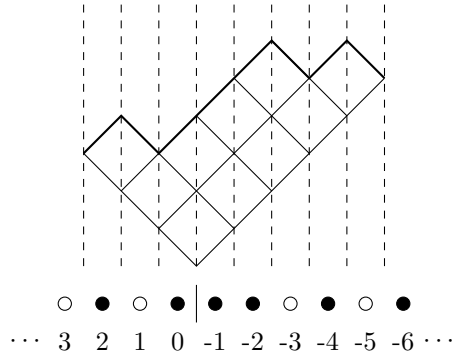
\begin{figure}[h]
\centering
\begin{tikzpicture}[scale=.5]
\draw[thick] (-2.5,5)--(-1.5,6)--(-0.5,5)--(2.5,8)--(3.5,7)--(4.5,8)--(5.5,7);;
\draw (-1.5,4)--(-.5,5);;
\draw (-.5,3)--(3.5,7);;
\draw (-2.5,5)--(.5,2)--(5.5,7);;
\draw (-.5,5)--(1.5,3);;
\draw (.5,6)--(2.5,4);;
\draw (1.5,7)--(3.5,5);;
\draw (2.5,8)--(4.5,6);;
\draw (-4, 0) node {$\cdots$};;
\draw (-3,0) node {3};;
\draw (-3,1) circle (5pt);;
\draw (-2,0) node {2};;
\draw[fill=black] (-2,1) circle (5pt);;
\draw (-1,0) node {1};;
\draw (-1,1) circle (5pt);;
\draw (0,0) node {0};;
\draw[fill=black] (0,1) circle (5pt);;
\draw (.5,.5)--(.5,1.5);;
\draw (1,0) node {-1};;
\draw[fill=black] (1,1) circle (5pt);;
\draw (2,0) node {-2};;
\draw[fill=black] (2,1) circle (5pt);;
\draw (3,0) node {-3};;
\draw (3,1) circle (5pt);;
\draw (4,0) node {-4};;
\draw[fill=black] (4,1) circle (5pt);;
\draw (5,0) node {-5};;
\draw (5,1) circle (5pt);;
\draw (6,0) node {-6};;
\draw[fill=black] (6,1) circle (5pt);;
\draw (7,0) node {$\cdots$};;
\draw[dashed] (0.5,2)--(0.5,9);;
\draw[dashed] (1.5,2)--(1.5,9);;
\draw[dashed] (2.5,2)--(2.5,9);;
\draw[dashed] (3.5,2)--(3.5,9);;
\draw[dashed] (4.5,2)--(4.5,9);;
\draw[dashed] (-0.5,2)--(-0.5,9);;
\draw[dashed] (-1.5,2)--(-1.5,9);;
\draw[dashed] (-2.5,2)--(-2.5,9);;
\draw[dashed] (5.5,2)--(5.5,9);;
\end{tikzpicture}
\caption{The Young-Maya correspondence for $(5,4,1)$.}
\label{fig:YoungMayaEx}
\end{figure}

\begin{prop}
The Young-Maya correspondence gives a bijection between partitions and Maya diagrams with charge zero.
\end{prop}

\begin{rem}\label{MayaRem}
Observe that adding a node to $\lambda$ corresponds to the following switch on adjacent beads in $m(\lambda)$:

\vspace{.1in}

\centerline{\begin{tikzpicture}[scale=.5]
\draw (0,1) circle (5pt);;
\draw[fill=black] (1,1) circle (5pt);;
\draw (3,1) node {$\rightarrow$};;
\draw (6,1) circle (5pt);;
\draw[fill=black] (5,1) circle (5pt);;
\end{tikzpicture}}

\vspace{.1in}

\noindent On the other hand, removing a node corresponds to the opposite switch:
\vspace{.1in}

\centerline{\begin{tikzpicture}[scale=.5]
\draw[fill=black]  (0,1) circle (5pt);;
\draw (1,1) circle (5pt);;
\draw (3,1) node {$\rightarrow$};;
\draw[fill=black] (6,1) circle (5pt);;
\draw (5,1) circle (5pt);;
\end{tikzpicture}}

\vspace{.1in} 

\noindent Notice now that doing the same swaps above in $m(\lambda)$ to beads indexed at $n$ and $n+k$ instead adds/removes $k$-strips to/from $\lambda$.
\end{rem}

\subsubsection{Core-quotient decomposition}\label{CoreQuotDec}
For a partition $\lambda$ and its Maya diagram $m(\lambda)$, consider its \textit{quotient Maya diagrams}
\[m_i(\lambda)(n)=m(\lambda)(i+n\ell)\]
for $0\le i\le\ell-1$.
In general, $m_i(\lambda)$ will have a charge $c_i$ that may not be zero.
If we draw a notch immediately right of the bead indexed by $-c_i$ in $m_i(\lambda)$, then we can perform the Young-Maya correspondence where the content zero line of the Young diagram is lined up with this notch instead.
This yields a partition $\lambda^i$.
We call
\[\quot(\lambda):=(\lambda^0,\ldots,\lambda^{\ell-1})\]
the \textit{$\ell$-quotient}.
The example of $\ell=3$ and $\lambda=(5,4,1)$ is show in Figure \ref{fig:quotientex}.
Applying Remark \ref{MayaRem} to $m_i(\lambda)$ and then to $m(\lambda)$, we can see that each node of $m_i(\lambda)$ corresponds to an $\ell$-strip in $m(\lambda)$ whose initial node is an $i$-node. 
Thus, $\quot(\lambda)$ records the $\ell$-strips of $\lambda$.

\begin{figure}[h]
\centering
\begin{tikzpicture}[scale=.5]
\draw (-8,.5) node {$m(\lambda)$:};;
\draw (-6, 0) node {$\cdots$};;
\draw (-5,0) node {3};;
\draw (-5,1) circle (5pt);;
\draw (-4,0) node {2};;
\draw[fill=black] (-4,1) circle (5pt);;
\draw (-3,0) node {1};;
\draw (-3,1) circle (5pt);;
\draw (-2,0) node {0};;
\draw[fill=black] (-2,1) circle (5pt);;
\draw (-1.5,.5)--(-1.5,1.5);;
\draw (-1,0) node {-1};;
\draw[fill=black] (-1,1) circle (5pt);;
\draw (0,0) node {-2};;
\draw[fill=black] (0,1) circle (5pt);;
\draw (1,0) node {-3};;
\draw (1,1) circle (5pt);;
\draw (2,0) node {-4};;
\draw[fill=black] (2,1) circle (5pt);;
\draw (3,0) node {-5};;
\draw (3,1) circle (5pt);;
\draw (4,0) node {-6};;
\draw[fill=black] (4,1) circle (5pt);;
\draw (5,0) node {$\cdots$};;
\draw (-15,-5.5) node {$m_0(\lambda):$};;
\draw (-13,-6) node {$\cdots$};;
\draw (-12,-6) node {3};;
\draw (-12,-5) circle (5pt);;
\draw (-11,-6) node {0};;
\draw[fill=black] (-11,-5) circle (5pt);;
\draw (-10,-6) node {-3};;
\draw (-10,-5) circle (5pt);;
\draw (-10.5, -4)--(-9.5,-3)--(-10.5,-2)--(-11.5,-3)--(-10.5,-4);;
\draw (-9,-6) node {-6};;
\draw[fill=black] (-9,-5) circle (5pt);;
\draw (-8,-6) node {$\cdots$};;
\draw (-10.5,-5.5)--(-10.5,-4.5);;
\draw (-5,-5.5) node {$m_1(\lambda):$};;
\draw (-3,-6) node {$\cdots$};;
\draw (-2,-6) node {1};;
\draw (-1.5,-5.5)--(-1.5,-4.5);;
\draw (-2,-5) circle (5pt);;
\draw (-1,-6) node {-2};;
\draw[fill=black] (-1,-5) circle (5pt);;
\draw[dashed] (-.5,-5.5)--(-.5,-4.5);;
\draw (-.5, -4)--(.5,-3)--(-.5,-2)--(-1.5,-3)--(-.5,-4);;
\draw (0,-6) node {-5};;
\draw (0,-5) circle (5pt);;
\draw (1,-6) node {-8};;
\draw[fill=black] (1,-5) circle (5pt);;
\draw (2,-6) node {$\cdots$};;
\draw (5, -5.5) node {$m_2(\lambda):$};;
\draw (7,-6) node {$\cdots$};;
\draw (8,-6) node {5};;
\draw (8,-5) circle (5pt);;
\draw[dashed] (8.5,-5.5)--(8.5,-4.5);;
\draw (8.5,-3) node {$\varnothing$};;
\draw (9,-6) node {2};;
\draw[fill=black] (9,-5) circle (5pt);;
\draw (9.5,-5.5)--(9.5,-4.5);;
\draw (10,-6) node {$\cdots$};;
\end{tikzpicture}
\caption{Computing the 3-quotient of $(5,4,1)$. We index the beads in $m_i(\lambda)$ with their original indices in $m(\lambda)$. The solid line is the central line, whereas the dashed line is the notch right of $-c_i$.}
\label{fig:quotientex}
\end{figure}
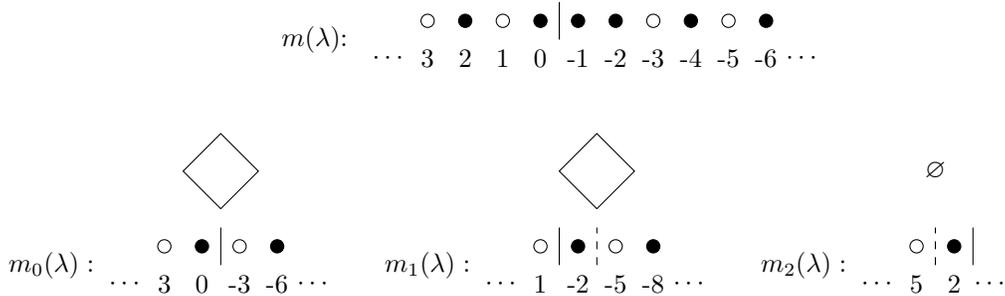

Once all the $\ell$-strips are removed from $\lambda$, we are left with its \textit{$\ell$-core} $\core(\lambda)$.
An \textit{$\ell$-core partition} is one that has no $\ell$-strips. 
To obtain the $\ell$-core from $m(\lambda)$, we change each $m_i(\lambda)$ into the diagram where all beads $<-c_i$ are black and all beads $\ge -c_i$ are white.
This is like the vacuum diagram except that the central line is replaced by our notch right of $-c_i$. 
Performing corresponding changes in the total Maya diagram $m(\lambda)$, we obtain $\core(\lambda)$.
This is illustrated for $\ell=3$ and $\lambda=(5,4,1)$ in Figure \ref{fig:coreex}.
From this, it follows that $\core(\lambda)$ is determined by the charges $c_i$.
Since $m(\lambda)$ has charge zero, it follows that
\[c_0+\cdots+c_{\ell-1}=0\]
Thus, we can view $(c_0,\ldots, c_{\ell-1})$ as a vector in the $A_{\ell-1}$ root lattice $Q$, and in fact all root lattice vectors can be realized in this way.
We will abuse notation and also denote this vector as $\core(\lambda)$.
This assignment $\lambda\mapsto (\core(\lambda), \quot(\lambda))$ is called the \textit{core-quotient decomposition} of $\lambda$.
We have the following:

\begin{prop}[\cite{JamesComb}]
The core-quotient decomposition yields a bijection
\[\{\hbox{\rm partitions}\}\leftrightarrow\{\hbox{\rm$\ell$-core partitions}\}\times\left\{ \hbox{\rm$\ell$-multipartitions} \right\}\]
\end{prop}

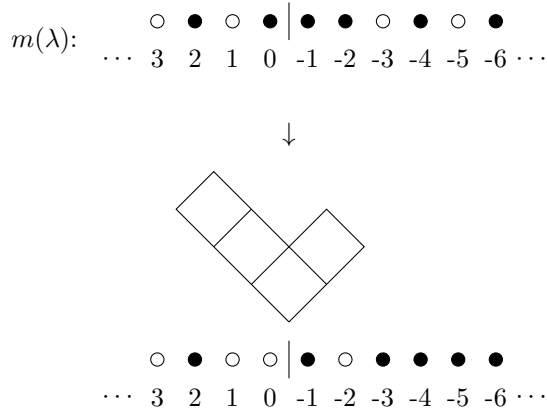
\begin{figure}[h]
\centering
\begin{tikzpicture}[scale=.5]
\draw (-6,.5) node {$m(\lambda)$:};;
\draw (-4, 0) node {$\cdots$};;
\draw (-3,0) node {3};;
\draw (-3,1) circle (5pt);;
\draw (-2,0) node {2};;
\draw[fill=black] (-2,1) circle (5pt);;
\draw (-1,0) node {1};;
\draw (-1,1) circle (5pt);;
\draw (0,0) node {0};;
\draw[fill=black] (0,1) circle (5pt);;
\draw (.5,.5)--(.5,1.5);;
\draw (1,0) node {-1};;
\draw[fill=black] (1,1) circle (5pt);;
\draw (2,0) node {-2};;
\draw[fill=black] (2,1) circle (5pt);;
\draw (3,0) node {-3};;
\draw (3,1) circle (5pt);;
\draw (4,0) node {-4};;
\draw[fill=black] (4,1) circle (5pt);;
\draw (5,0) node {-5};;
\draw (5,1) circle (5pt);;
\draw (6,0) node {-6};;
\draw[fill=black] (6,1) circle (5pt);;
\draw (7,0) node {$\cdots$};;
\draw (0.5,-2) node {$\downarrow$};;
\draw (-4, -9) node {$\cdots$};;
\draw (-3,-9) node {3};;
\draw (-3,-8) circle (5pt);;
\draw (-2,-9) node {2};;
\draw[fill=black] (-2,-8) circle (5pt);;
\draw (-1,-9) node {1};;
\draw (-1,-8) circle (5pt);;
\draw (0,-9) node {0};;
\draw (0,-8) circle (5pt);;
\draw (.5,-8.5)--(.5,-7.5);;
\draw (1,-9) node {-1};;
\draw[fill=black] (1,-8) circle (5pt);;
\draw (2,-9) node {-2};;
\draw (2,-8) circle (5pt);;
\draw (3,-9) node {-3};;
\draw[fill=black] (3,-8) circle (5pt);;
\draw (4,-9) node {-4};;
\draw[fill=black] (4,-8) circle (5pt);;
\draw (5,-9) node {-5};;
\draw[fill=black] (5,-8) circle (5pt);;
\draw (6,-9) node {-6};;
\draw[fill=black] (6,-8) circle (5pt);;
\draw (7,-9) node {$\cdots$};;
\draw (.5,-7)--(2.5,-5)--(1.5,-4)--(.5,-5)--(-1.5,-3)--(-2.5,-4)--(.5,-7);;
\draw (1.5,-6)--(.5,-5);;
\draw (-.5,-6)--(.5,-5);;
\draw (-1.5,-5)--(-.5,-4);;
\end{tikzpicture}
\caption{The 3-core of $\lambda=(5,4,1)$.}
\label{fig:coreex}
\end{figure}

\subsubsection{Orders on columns and rows}\label{Orders}
To conclude this subsection, we will review an ordering on the columns and rows of $\quot(\lambda)$ introduced in \cite{WreathEigen}.
First observe that every Maya diagram has a maximal infinite sequence of consecutive black beads and a maximal infinite sequence of consecutive white beads; we call them the \textit{black sea} and \textit{white sea}, respectively. 
By Remark \ref{MayaRem}, adding a column of length $k$ to $\mu^i$ that is no longer than any of the existing columns of $\mu^i$ entails the following swap in $m_i(\mu)$.

\vspace{.1in}

\centerline{\begin{tikzpicture}[scale=.5]
\draw (0,1) circle (5pt);;
\draw (0,0) node {$j+k$};;
\draw (1,1) circle (5pt);;
\draw (2,1) node {$\cdots$};;
\draw (3,1) circle (5pt);;
\draw (4,0) node {$j$};;
\draw[fill=black] (4,1) circle (5pt);;
\draw (6,1) node {$\rightarrow$};;
\draw (8,0) node {$j+k$};;
\draw[fill=black] (8,1) circle (5pt);;
\draw (9,1) circle (5pt);;
\draw (10,1) node {$\cdots$};;
\draw (11,1) circle (5pt);;
\draw (12,0) node {$j$};;
\draw (12,1) circle (5pt);;
\end{tikzpicture}}

\vspace{.1in}

\noindent Here, all the beads in the middle are white and the black bead is the leftmost bead in the black sea of $m_i(\mu)$.
Putting this in the total Maya diagram, one can see that we have added a $k\ell$-strip to $\mu$ wherein the consecutive $i$ and $(i+1)$ nodes are vertically adjacent.
Similarly, adding a row of length $k$ to $\mu^i$ that is no longer than any of the previous nodes entails swapping the rightmost bead in the white sea of $m_i(\mu)$ with a black bead $k$ places to the right, and this is equivalent to adding a $k\ell$-strip to $\mu$ wherein consecutive $i$ and $(i+1)$-nodes are horizontally adjacent.

For a partition $\lambda$, we define the \textit{left-to-right order} on the columns of $\quot(\lambda)$ to be the unique order such that if we perform the corresponding strip additions to $\core(\lambda)$, the initial nodes for earlier columns are northwest of the initial nodes for later columns.
The existence of this order comes from our observation above: to each column of $\lambda^i$, we can assign a black bead that came from the black sea, and we are merely stipulating that, viewing these beads in $m(\lambda)$, those of earlier columns are left of those of later columns.
In the example of of the 3-quotient of $\lambda=(5,4,1)$ depicted in Figure \ref{fig:quotientex}, we have that the box in $\lambda^0$ comes before the box in $\lambda^1$.
Similarly, the \textit{right-to-left order} on the rows of $\quot(\lambda)$ is the unique one in which performing the strip additions to $\core(\lambda)$ in that order results in the final nodes of earlier rows lying southeast of those of later rows. 
We note that while we call these orderings on the columns and rows of $\quot(\lambda)$, they are defined in terms of the original partition $\lambda$.

\subsection{Wreath Macdonald polynomials}
With the combinatorics of $\ell$-cores and $\ell$-quotients reviewed, we are ready to define the transformed and ordinary wreath Macdonald polynomials.

\subsubsection{Plethysm}
Now, for two indeterminates $q$ and $t$, let 
\begin{align*}
\Lambda_{q,t}&:=\CC(q,t)\otimes\Lambda\\
R_{q,t}&:=\CC(q,t)\otimes R
\end{align*}
The group $\Gamma_n$ has a natural reflection representation $\mathfrak{h}_n$.
For an indeterminate $s$, we define the operators $\bigwedge_s^+$ and $\bigwedge_s^-$ on $\CC(s)\otimes R$ such that
\begin{align*}
\left.\textstyle\bigwedge_s^+\right|_{R_n}&=(-)\otimes\sum_{i=0}^n(-s)^i\left[\textstyle\bigwedge^i\mathfrak{h}_n\right]\\
\left.\textstyle\bigwedge_s^-\right|_{R_n}&=(-)\otimes\sum_{i=0}^n(-s)^i\left[\textstyle\bigwedge^i\mathfrak{h}_n^*\right]
\end{align*}
The following was proved in \cite{WreathEigen}:

\begin{prop}
$\bigwedge_s^\pm$ is the ring automorphisms on $\CC(s)\otimes\Lambda^{\otimes\ell}$ given by
\begin{equation}
\textstyle\bigwedge_s^\pm f\left[ X^{\bullet} \right]=f\left[ (1-s\sigma^{\pm 1})X^{\bullet} \right]\label{pleths}
\end{equation}
\end{prop}

\subsubsection{Transformed polynomials}\label{TransPol}
For $\lambda,\mu\vdash n$, we let $\lambda \ge_\ell\mu$ denote that $\lambda\ge\mu$ and $\core(\lambda)=\core(\mu)$.

\begin{defn}\label{WreathDef}
For $\lambda$ with $\quot(\lambda)\vdash n$, the \textit{transformed wreath Macdonald polynomial} $H_\lambda(q,t)=H_\lambda\left[ X^\bullet; q,t \right]$ is the element of $\Lambda_{q,t}^{\otimes\ell}$ characterized by
\begin{enumerate}
\item $H_\lambda\left[ (1-q\sigma^{-1})X^\bullet ;q,t \right]$ lies in the span of $\{s_{\quot(\mu)} :\mu\ge_\ell\lambda\}$;
\item $H_\lambda\left[ (1-t^{-1}\sigma^{-1})X^\bullet; q,t \right]$ lies in the span of $\{s_{\quot(\mu)} :\mu\le_\ell\lambda\}$;
\item the coefficient of the trivial representation of $\Gamma_n$ is 1.
\end{enumerate}
When appropriate, we will abbreviate it by $H_\lambda$.
\end{defn}
\noindent The transformed wreath Macdonald polynomials are the ones defined in \cite{Haiman} and studied in our previous work \cite{WreathEigen}.
Their existence was proved by Bezrukavnikov and Finkelberg \cite{BezFink}.
We have rewritten the definition in terms of matrix plethysms using (\ref{pleths}).
%
%

\subsubsection{Sectors}\label{Sectors}
Recall from \ref{CoreQuotDec} that $Q$ denotes the $A_{\ell-1}$ root lattice, and we can use it to index $\ell$-cores.
For $\alpha\in Q$, we denote the corresponding basis element in $\CC[Q]$ by $e^\alpha$.
We will work in the ring
\[\Lambda^{\otimes\ell}\otimes\CC[Q]\]
\begin{prop}
$\{H_\lambda\otimes e^{\core(\lambda)}\}$ is a basis of $\Lambda^{\otimes\ell}\otimes\CC[Q]$.
\end{prop}
\noindent
Finally, if $\alpha=(a_0,\ldots, a_{\ell-1})$, then we define
\[\tensor[^t]{\alpha}{}=(-a_{\ell-1},-a_{\ell-2},\ldots,-a_1,-a_0 )\]
This operation is defined so that if $\core(\lambda)=\alpha$, then $\core(\tensor[^t]{\lambda}{})=\tensor[^t]{\alpha}{}$.

\subsubsection{Transformed pairing}
Next, we consider a pairing under which $\{H_\lambda\}$ is some kind of orthogonal basis.
Recall that the Hall pairing $\langle -,-\rangle$ corresponds to the Hom pairing in $R$.
Alternatively, we can twist with $\iota$ to obtain the pairing corresponding to tensoring and taking invariants:
\[
\langle f,g\rangle^*:=\langle {\textstyle f\left[ \iota X^\bullet \right]}, g\rangle
\]
Note that under $\langle -,-\rangle^*$, the shift plethysm $\sigma$ is self-adjoint.
On the other hand, under $\langle-,-\rangle$, $\sigma$ is adjoint to $\sigma^{-1}$.
The \textit{transformed wreath Macdonald pairing} is defined as follows:
\begin{align*}
\left\langle f,g\right\rangle_{q,t}'&:=\left\langle\textstyle f,g\left[ \sigma(1-t^{-1}\sigma^{-1})(1-q\sigma^{-1})X^\bullet \right]\right\rangle^*\\
&=\left\langle\textstyle f\left[ \iota X^\bullet \right],g\left[ \sigma(1-t^{-1}\sigma^{-1})(1-q\sigma^{-1})X^\bullet \right]\right\rangle
\end{align*}
\textit{A priori}, this is only defined on $\Lambda_{q,t}^{\otimes\ell}$.
We extend it to $\Lambda_{q,t}^{\otimes\ell}\otimes\CC[Q]$ by setting
\[\langle f\otimes e^\alpha,g\otimes e^{\beta}\rangle_{q,t}'=\delta_{\alpha,\tensor[^t]{\beta}{}}\langle f,g\rangle_{q,t}'\]
where $\tensor[^t]{\beta}{}$ was defined in \ref{Sectors}.
To find a dual basis to $\{H_\lambda\otimes e^{\core(\lambda)}\}$, we introduce the following \textit{orthogonal variant}:
\[
\textstyle H_\lambda^*\left[X^\bullet;q,t\right]:=H_\lambda\left[-X^\bullet;t^{-1},q^{-1}  \right]
\]
As with $H_\lambda$, we abbreviate it by $H_\lambda^*$.

\begin{prop}\label{PairProp}
The pairing $\langle H_{\tensor[^t]{\lambda}{}}^*,H_\mu\rangle_{q,t}'$ is nonzero if and only if $\lambda=\mu$.
\end{prop}

\begin{proof}
If $\lambda\not\ge\mu$, then by adjunction, we have
\begin{align}
\nonumber
\langle H_{\tensor[^t]{\lambda}{}}^*,H_\mu\rangle_{q,t}'&= \left\langle\textstyle H^*_{\tensor[^t]{\lambda}{}}, H_\mu\left[ \sigma(1-t^{-1}\sigma^{-1})(1-q\sigma^{-1})X^\bullet \right]\right\rangle^*\\
\nonumber
&=\left\langle\textstyle H^*_{\tensor[^t]{\lambda}{}}\left[\sigma(1-t^{-1}\sigma^{-1}) X^\bullet \right], H_\mu\left[ (1-q\sigma^{-1})X^\bullet \right]\right\rangle^*\\
\label{PairingRewrite}
&= \left\langle\textstyle H_{\tensor[^t]{\lambda}{}}\left[-\iota\sigma(1-t^{-1}\sigma^{-1}) X^\bullet;t^{-1},q^{-1} \right], H_\mu\left[ (1-q\sigma^{-1})X^\bullet; q,t \right]\right\rangle
\end{align}
By the first triangularity condition in Definition \ref{WreathDef}, the second argument of (\ref{PairingRewrite}) is a linear combination of $s_{\quot(\nu)}$ for $\nu\ge_\ell \mu$.
Similarly, the first argument of (\ref{PairingRewrite}) is a linear combination of $s_{\quot(\tensor[^t]{\nu}{})}\left[ -\iota \sigma X^\bullet \right]$ for $\tensor[^t]{\nu}{}\ge_\ell\tensor[^t]{\lambda}{}$. 
The key observation is that
\[
\textstyle s_{\quot(\tensor[^t]{\nu}{})}\left[ -\iota\sigma X^\bullet \right]=(-1)^{|\quot(\nu)|}s_{\quot(\nu)}
\]
To see this, note that $\iota\sigma$ will send color $i$ to color $-i-1$, while transposing an $\ell$-strip starting at an $i$-node yields an $\ell$-strip starting at an $-i-1$ node.
From this, it follows that the pairing (\ref{PairingRewrite}) is zero when $\lambda\not\ge\mu$.
The case $\lambda\not\le\mu$ is similar, and thus the pairing is always zero when $\lambda\not=\mu$.
Finally, because $\left\{ H_\lambda \otimes e^{\core(\lambda)}\right\}$ and $\left\{ H_\lambda^*\otimes e^{\core(\lambda)} \right\}$ are bases and $\langle -,-\rangle'_{q,t}$ is nondegenerate, the pairing for $\lambda=\mu$ must be nonzero.
%
\end{proof}

\begin{rem}
A cleaner approach would be to define $\langle -,-\rangle_{q,t}'$ without the extra twist by $\iota\sigma$ and move it into the definition of $H_\lambda^*$.
Namely, one can check that $H_{\tensor[^t]{\lambda}{}}\left[ -\iota\sigma X^\bullet; t^{-1},q^{-1} \right]$ satisfies the triangularity conditions of Definition \ref{WreathDef} but with $\sigma^{-1}$ replaced by $\sigma$:
\begin{enumerate}
\item $H_{\tensor[^t]{\lambda}{}}\left[ (1-q\sigma)(-\iota\sigma X^\bullet) ;t^{-1},q^{-1} \right]$ lies in the span of $\{s_{\quot(\mu)} :\mu\ge_\ell\lambda\}$;
\item $H_{\tensor[^t]{\lambda}{}}\left[ (1-t^{-1}\sigma)(-\iota\sigma X^\bullet); t^{-1},q^{-1} \right]$ lies in the span of $\{s_{\quot(\mu)} :\mu\le_\ell\lambda\}$.
\end{enumerate}
Thus, we can instead use these altered triangularity conditions (along with a natural normalization condition) to define the positive variants.
While this approach is more natural, it leads to ``one twist/swap too many'' for us (the author) to handle when we include the quantum toroidal algebra later on.
\end{rem}

\subsubsection{Ordinary polynomials}
Recall that for $\ell=1$, the \textit{ordinary} Macdonald polynomial $P_\lambda$ is obtained from $H_\lambda(q,t^{-1})$ by renormalizing $H_\lambda\left[ (1-t)X ; q,t^{-1}\right]$ so that the leading term $s_{\lambda}$ has coefficient $1$.
The second triangularity condition (2) of Definition \ref{WreathDef} ensures that the latter is upper triangular with respect to the Schur function basis ordered by dominance.
We similarly define the ordinary polynomials for general $\ell$.
Note that the proof of Proposition \ref{PairProp} implies that the coefficient of $s_{\quot(\lambda)}$ in $H_\lambda\left[ (1-t^{-1}\sigma^{-1})X^\bullet \right]$ is nonzero.

\begin{defn}\label{OrdDef}
The \textit{ordinary wreath Macdonald polynomial} $P_\lambda(q,t)=P_\lambda\left[ X^\bullet ;q,t \right]$ is the renormalization of $H_\lambda\left[ (1-t\sigma^{-1})X^\bullet ; q,t^{-1} \right]$ such that the coefficient of $s_{\quot(\lambda)}$ is $1$.
We also set $P_\lambda^*(q,t):=P_\lambda\left[ -X^\bullet ;t,q \right]$.
\end{defn}

Inspired by the $\ell=1$ case, we define the (ordinary) \textit{wreath Macdonald pairing} by
\[
\langle f, g\rangle_{q,t}:=\left\langle f ,  g\left[ \frac{\sigma(1-q\sigma^{-1})}{(1-t\sigma^{-1})}X^\bullet \right]\right\rangle^*
\]
We extend it to $\Lambda^{\otimes\ell}\otimes\CC[Q]$ like we did with $\langle -,-\rangle_{q,t}'$, namely by making $e^\alpha$ dual to $e^{\tensor[^t]{\alpha}{}}$.
The \textit{dual (ordinary) wreath Macdonald polynomial} $Q_\lambda(q,t)$ is the renormalization of $H_\lambda\left[ (1-t\sigma^{-1})X^\bullet;q,t^{-1} \right]$ such that the coefficient of $s_{\quot(\lambda)}$ in
\[
Q_\lambda\left[ \frac{(1-q\sigma^{-1})}{(1-t\sigma^{-1})}X^\bullet \right]
\]
is $1$.
Finally, we set $Q^*_\lambda:=Q_\lambda\left[ -X^\bullet;t,q \right]$.
As before, we will drop $(q,t)$ for $P_\lambda$, $P^*_\lambda$, $Q_\lambda$, and $Q^*_\lambda$ if it causes no confusion, although care is certainly necessary as these parameters are often swapped and inverted.
The following is proved much like Proposition \ref{PairProp} (note the transpose):

\begin{prop}
We have
\[
\langle P_{\tensor[^t]{\lambda}{}}^*, Q_\mu\rangle_{q,t}=\langle Q_{\tensor[^t]{\lambda}{}}^*, P_\mu\rangle_{q,t}=\delta_{\lambda,\mu}
\]
\end{prop}

\noindent In this paper, we will not work directly with $P_\lambda$ or $Q_\lambda$, but rather with
\begin{align*}
\tilde{P}_\lambda[X^\bullet;q,t]&:=P_\lambda\left[ (1-t\sigma^{-1})^{-1}X^\bullet;q,t \right]\\
\tilde{Q}_\lambda[X^\bullet;q,t]&:=Q_\lambda\left[ (1-t\sigma^{-1})^{-1}X^\bullet;q,t \right]
\end{align*}
Note that $\tilde{P}_\lambda(q,t^{-1})$ and $\tilde{Q}_\lambda(q,t^{-1})$ are scalar multiples of $H_\lambda$.

We end by stating a natural wreath generalization of the Macdonald norm formula (\ref{NormForm}) that has been confirmed by computations of Orr and Shimozono:
\begin{conj}[Conjecture 3.36 of \cite{OSWreath}]\label{NormConj}
For a node $\square=(a,b)\in\lambda$, let:
\begin{itemize}
\item $a(\square)=\lambda_{b}-a$ denote its arm length;
\item $\ell(\square)=\tensor[^t]{\lambda}{}_a-b$ denote its leg length;
\item $h(\square)=a(\square)+\ell(\square)+1$ denote its hook length.
\end{itemize} 
We then have the norm formula:
\[
\langle P_{\tensor[^t]{\lambda}{}}^*,P_\lambda\rangle_{q,t}=\prod_{\substack{\square\in\lambda\\ h(\square)\equiv 0\,\hbox{\tiny{\rm mod $\ell$} }}}
\frac{1-q^{a(\square)+1}t^{\ell(\square)}}{1-q^{a(\square)}t^{\ell(\square)+1}}
\]
\end{conj}

\section{Quantum toroidal and shuffle algebras}
In this section, $\ell\ge 3$.
Let $\qqq,\ddd$ be indeterminates and set $\FF:=\CC(\qqq,\ddd)$---we will specify their relationship to $q$ and $t$ later on.
We will make use of the delta function $\delta(z)$:
\[\delta(z)=\sum_{n\in\ZZ}z^n\]

\subsection{Quantum toroidal algebra}
For $i,j\in\ZZ/\ell\ZZ$, we let 
\[a_{i,i}=2,\, a_{i,i\pm1}=-1,\, m_{i,i\pm 1}=\mp 1,\hbox{ and }a_{i,j}=m_{i,j}=0\hbox{ otherwise}\]
Also, define
\[g_{i,j}(z):=\frac{\qqq^{a_{i,j}}z-1}{z-\qqq^{a_{i,j}}}\]

\subsubsection{Definition}
The \textit{quantum toroidal algebra} $\UTor$ is a $\FF$-algebra with generators
\[\{e_{i,n}, f_{i,n},\psi_{i,n},\psi_{i,0}^{-1},\gamma^{\pm\frac{1}{2}},\qqq^{d_1},\qqq^{d_2}\}_{i\in\ZZ/\ell\ZZ}^{n\in\ZZ}\]
To describe its relations, we define the following generating series of generators, called \textit{currents}:
\begin{align*}
e_i(z)&=\sum_{n\in\ZZ}e_{i,n}z^{-n}\\
f_i(z)&=\sum_{n\in\ZZ}f_{i,n}z^{-n}\\
\psi_i^\pm(z)&=\psi_{i,0}^{\pm1}+\sum_{n>0}\psi_{i,\pm n}z^{\mp n}
\end{align*}
The relations are then:
\begin{gather*}
[\psi_i^\pm(z),\psi_j^\pm(w)]=0,\,\gamma^{\pm\frac{1}{2}}\hbox{ are central},\\
\psi_{i,0}^{\pm1}\psi_{i,0}^{\mp1}=\gamma^{\pm\frac{1}{2}}\gamma^{\mp\frac{1}{2}}=\qqq^{\pm d_1}\qqq^{\mp d_1}=\qqq^{\pm d_2}\qqq^{\mp d_2}=1,\\
\qqq^{d_1}e_i(z)\qqq^{-d_1}=e_i(\qqq z),\, \qqq^{d_1}f_i(z)\qqq^{-d_1}=f_i(\qqq z),\, \qqq^{d_1}\psi_i^\pm(z)\qqq^{-d_1}=\psi_i^\pm(\qqq z),\\
\qqq^{d_2}e_i(z)\qqq^{-d_2}=\qqq e_i(z),\, \qqq^{d_2}f_i(z)\qqq^{-d_2}=\qqq^{-1} f_i(z),\, \qqq^{d_2}\psi_i^\pm(z)\qqq^{-d_2}=\psi_i^\pm( z),\\
g_{i,j}(\gamma^{-1}\ddd^{m_{i,j}}z/w)\psi_i^{+}(z)\psi_j^{-}(w)=g_{i,j}(\gamma\ddd^{m_{i,j}} z/w)\psi_j^{-}(w)\psi_i^{+}(z),\\
e_i(z)e_j(w)=g_{i,j}(\ddd^{m_{i,j}}z/w)e_j(w)e_i(z),\\
f_i(z)f_j(w)=g_{i,j}(\ddd^{m_{i,j}}z/w)^{-1}f_j(w)f_i(z),\\
(\qqq-\qqq^{-1})[e_i(z),f_j(w)]=\delta_{i,j}\left(\delta(\gamma w/z)\psi_i^+(\gamma^{\frac{1}{2}}w)-\delta(\gamma z/w)\psi_i^-(\gamma^\frac{1}{2}z)\right),\\
\psi_i^\pm(z)e_j(w)=g_{i,j}(\gamma^{\pm\frac{1}{2}}\ddd^{m_{i,j}}z/w)e_j(w)\psi_i^\pm(z),\\
\psi_i^\pm(z)f_j(w)=g_{i,j}(\gamma^{\mp\frac{1}{2}}\ddd^{m_{i,j}}z/w)^{-1}f_j(w)\psi_i^\pm(z),\\
\Sym_{z_1,z_2}[e_i(z_1),[e_i(z_2),e_{i\pm1}(w)]_\qqq]_{\qqq^{-1}}=0,\,[e_i(z),e_j(w)]=0\hbox{ for }j\not=i,i\pm1,\\
\Sym_{z_1,z_2}[f_i(z_1),[f_i(z_2),f_{i\pm1}(w)]_\qqq]_{\qqq^{-1}}=0,\,[f_i(z),f_j(w)]=0\hbox{ for }j\not=i,i\pm1,
\end{gather*}
In the above, $[a,b]_\qqq=ab-\qqq ba$ is the $\qqq$-commutator.
We will also need the elements $\{b_{i,n}\}_{i\in\ZZ/\ell\ZZ}^{n\not=0}$ defined by
\[\psi_{i}^\pm(z)=\psi_{i,0}^{\pm1}\exp\left( \pm(\qqq-\qqq^{-1})\sum_{n>1}b_{i,\pm n}z^{\mp n} \right)\]
Denote by $\Heis$ the subalgebra generated by $\{b_{i,n}\}\cup\{\gamma^{\pm\frac{1}{2}}\}$.

\subsubsection{Miki automorphism}
One feature making the quantum toroidal algebra deserving of its name is that it contains two copies of the quantum affine algebra $U_\qqq(\dot{\mathfrak{sl}}_\ell)$.
The first, often called the \textit{horizontal subalgebra}, is generated by the zero modes:
\[
\{e_{i,0},f_{i,0}, \psi_{i,0}^{\pm 1}\}_{i\in\ZZ/\ell\ZZ}
\]
The second copy, the \textit{vertical subalgebra}, is roughly-speaking generated by currents where $i\not=0$:
\[
\{e_{i}(z),f_i(z)\}^{i\in\ZZ/\ell\ZZ}_{i\not=0}\cup\left\{ \psi_{i}^\pm(z) \right\}_{i\in\ZZ/\ell\ZZ}
\]
We will not go into detail on these subalgebras as they will not be used in the sequel.
However, they motivate the following automorphism of Miki, which swaps the ``two loops'' of the torus.
 
Let $'\ddot{U}$ denote the subalgebra of $\UTor$ obtained by omitting $\qqq^{d_1}$, $\ddot{U}'$ denote the subalgebra obtained by omitting $\qqq^{d_2}$, and $'\ddot{U}'$ denote the subalgebra obtained by omitting both elements.
Next, let $\eta$ denote the $\CC(\qqq)$-linear antiautomorphism of $'\ddot{U}'$ defined by
\begin{equation}
\begin{gathered}
    \eta(\ddd)=\ddd^{-1}\\
\eta(e_{i,k})=e_{i,-k},\, \eta(f_{i,k})=f_{i,-k},\, \eta(h_{i,k})=-\gamma^kh_{i,-k},\\
\eta(\psi_{i,0})=\psi_{i,0}^{-1},\,\eta(\gamma^{\frac{1}{2}})=\gamma^{\frac{1}{2}}
\end{gathered}
\label{EtaDef}
\end{equation}
We have the following beautiful result of Miki:
\begin{thm}[\cite{Miki}, \cite{Miki2}]\label{MikiAut}
There is an isomorphism $\varpi:\,'\ddot{U}\rightarrow\ddot{U}'$ sending the horizontal subalgebra to the vertical subalgebra.
Moreover, when restricted to $'\ddot{U}'$, we have:
\[
\varpi^{-1}=\eta\varpi\eta
\] 
\end{thm}

\subsection{Representations}
We will go over two representations of $\UTor$: the \textit{vertex} and \textit{Fock representations}.

\subsubsection{Heisenberg subalgebra}
Let $[n]_\qqq$ denote the quantum integer:
\[
[n]_\qqq:=\frac{\qqq^n-\qqq^{-n}}{\qqq-\qqq^{-1}}
\]
The subalgebra $\Heis$ satisfies the relation
\begin{equation}
[b_{i,n},b_{j,m}]=\delta_{m+n,0}\frac{\ddd^{-nm_{i,j}}[na_{i,j}]_{\qqq}(\gamma^n-\gamma^{-n})}{n(\qqq-\qqq^{-1})}
\label{HeisRel}
\end{equation}
%
Let
\begin{itemize}
\item $\Heis^+$ be the subalgebra generated by $\left\{ b_{i,n} \right\}_{i\in\ZZ/\ell\ZZ}^{n>0}$;
\item $\Heis^-$ be the subalgebra generated by $\left\{ b_{i,-n} \right\}_{i\in\ZZ/\ell\ZZ}^{n>0}$; 
\item and $\Heis^{\ge}$ be the subalgebra generated by $\left\{ b_{i,n} \right\}_{i\in\ZZ/\ell\ZZ}^{n>0}\cup\left\{ \gamma^{\pm\frac{1}{2}} \right\}$;
\end{itemize}
$\Heis^\ge$ has a one-dimensional representation $\FF\mathbb{1}$ where 
\begin{align*}
\gamma^{\pm\frac{1}{2}}\mathbb{1}&=\qqq^{\pm\frac{1}{2}}\mathbb{1}\\
b_{i,n}\mathbb{1}&=0
\end{align*} 
We then have
\[\mathrm{Ind}_{\Heis^\ge}^\Heis\FF\mathbb{1}\cong\CC[b_{i,-n}]_{i\in\ZZ/\ell\ZZ}^{n>0}\]
as $\Heis^-$-modules.
If we identify $\FF\otimes\Lambda^{\otimes\ell}$ with $\Heis^-$ by specifying
\begin{equation}
p_k(i)\mapsto \frac{k}{[k]_\qqq}b_{i,-k}
\label{PowertoBosons}
\end{equation}
then we also obtain an identification
\[\FF\otimes\Lambda^{\otimes\ell}\cong\mathrm{Ind}_{\Heis^\ge}^\Heis\FF\mathbb{1}\]

\subsubsection{Vertex representation}
We will now review what one can consider a homogeneous bosonic Fock space for $\UTor$.
For the $A_{\ell-1}$ root system, let $\{\alpha_j\}_{j=1}^{\ell-1}$ be the simple roots, let $\{h_j\}_{j=1}^{\ell-1}$ be the corresponding simple coroots, and let $\langle-,-\rangle$ denote the natural pairing between them. 
We set 
\begin{align*}
\alpha_0&:=-\sum_{j=1}^{\ell-1}\alpha_j,\\
h_0&:=-\sum_{j=1}^{\ell-1}h_j,
\end{align*} 
The \textit{twisted group algebra} $\FF\{Q\}$ is the $\FF$-algebra generated by $\{e^{\pm\alpha_j}\}_{j=1}^{\ell-1}$ satisfying the relations
\[e^{\alpha_i}e^{\alpha_j}=(-1)^{\langle h_i,\alpha_j\rangle}e^{\alpha_j}e^{\alpha_i}\]
For $\alpha\in Q$ with $\alpha=\sum_{j=1}^{\ell-1}m_j\alpha_j$, we set
\[e^{\alpha}=e^{m_1\alpha_1}\cdots e^{m_{\ell-1}\alpha_{\ell-1}}\]
We will identify $\FF\left[ Q \right]$ and $\FF\left\{ Q \right\}$ by the `identity map' $e^\alpha\mapsto e^\alpha$. 
Let
\[W:=\mathrm{Ind}_{\Heis^\ge}^\Heis\FF\mathbb{1}\otimes\FF\left\{ Q \right\}\]
Altogether, we have an identification
\begin{equation*}
\Lambda^{\otimes\ell}\otimes\FF[Q]\cong W
\label{VertexSym}
\end{equation*}

Finally, we will need to define some additional operators.
Let $v\otimes e^\alpha \in W$ where
\begin{align*}
v&=b_{i_1,-k_1}\cdots b_{i_N,-k_N} v_0\\
\alpha&=\sum_{j=1}^{\ell-1}m_j\alpha_j
\end{align*}
The operators $\partial_{\alpha_i}$, $z^{H_{i,0}}$, and $d$ are defined by
\begin{align*}
\partial_{\alpha_i}(v\otimes e^\alpha )&:=\langle h_i,\alpha\rangle v\otimes e^\alpha ,\\
z^{H_{i,0}}(v\otimes e^\alpha )&:=z^{\langle h_i,\alpha\rangle} \ddd^{\frac{1}{2}\sum_{j=1}^{\ell-1}\langle h_i,m_j\alpha_j\rangle m_{i,j}}v\otimes e^\alpha ,\\
d(v\otimes e^\alpha )&:=-\left(\dfrac{(\alpha,\alpha)}{2}+(\alpha,\Lambda_p)+\sum_{i=1}^Nk_i\right)v\otimes e^\alpha 
\end{align*}

The following theorem of Yoshihisa Saito defines the \textit{vertex representation}:

\begin{thm}[\cite{Saito}]\label{Vertexrep}
For any $\vec{\ccc}=(\ccc_0,\ldots, \ccc_{\ell-1})\in(\CC^\times)^\ell$, the following formulas endow $W$ with an action of $\ddot{U}'$
\begin{align*}
\rho_{\vec{\ccc}}(e_i(z))&=\ccc_i\exp\left(\sum_{k>0}\frac{\qqq^{-\frac{k}{2}}}{[k]_\qqq}b_{i,-k}z^k\right)\exp\left(-\sum_{k>0}\frac{\qqq^{-\frac{k}{2}}}{[k]_\qqq}b_{i,k}z^{-k}\right)e^{\alpha_i} z^{1+H_{i,0}},\\
\rho_{\vec{\ccc}}(f_i(z))&=\frac{(-1)^{\ell\delta_{i,0}}}{\ccc_i}\exp\left(-\sum_{k>0}\frac{\qqq^{\frac{k}{2}}}{[k]_\qqq}b_{i,-k}z^k\right)\exp\left(\sum_{k>0}\frac{\qqq^{\frac{k}{2}}}{[k]_\qqq}b_{i,k}z^{-k}\right)e^{-\alpha_i} z^{1-H_{i,0}},\\
\rho_{\vec{\ccc}}(\psi_i^\pm(z))&=\exp\left(\pm(\qqq-\qqq^{-1})\sum_{k>0}b_{i,\pm k}z^{\mp k}\right)\qqq^{\pm\partial_{\alpha_i}},\,\rho_{\vec{\ccc}}(\gamma^{\frac{1}{2}})=\qqq^\frac{1}{2},\,\rho_{\vec{\ccc}}(\qqq^{d_1})=\qqq^d
\end{align*}
The representation $W$ is irreducible.
\end{thm} 

\subsubsection{Fock representations}
The fermionic Fock space for $\UTor$ is often simply called the \textit{Fock representation} without ambiguity.
We denote its underlying vector space by $\mathcal{F}$.
It has a basis $\{|\lambda\rangle\}$ indexed by partitions.
We denote by $\langle\lambda|$ the dual element to $|\lambda\rangle$. 

Let $\square=(a,b)$ be a node which may or may not be in a partition $\lambda$.
We will denote by:
\begin{enumerate}
\item $\chi_\square=q^{a-1}t^{b-1}$ the character of the node;
\item $c_\square$ its content and $\bar{c}_\square$ its content mod $\ell$;
\item $d_i(\lambda)$ the number of $i$-nodes in $\lambda$;
\item and $A_i(\lambda)$ and $R_i(\lambda)$ the addable and removable $i$-nodes of $\lambda$, respectively.
\end{enumerate}
Finally, we will abbreviate $a\equiv b\hbox{ mod }\ell$ by simply $a\equiv b$ and use the Kronecker delta function 
\[\delta_{a\equiv b}=\left\{\begin{array}{ll}
1 & \hbox{if }a\equiv b\\
0 & \hbox{otherwise}
\end{array}\right.\]
In the proposition below, we consider a partition $\lambda$ and its enlargement $\lambda+\square$ into another partition by a single node $\square$.

\begin{prop}[\cite{NagaoK},\cite{FJMMRep},\cite{WreathEigen}]\label{FockRep} 
Let $\upsilon\in\FF^\times$.
If we identify
\[
q=\qqq\ddd,\, t=\qqq\ddd^{-1}
\]
we can define a $'\ddot{U}$-action $\tau_{\upsilon}^-$ on $\mathcal{F}$ where the only nonzero matrix elements of the generators are
\begin{gather*}
\begin{aligned}
\langle\lambda | \tau_{\upsilon}^-(e_i(z))|\lambda+\square\rangle&=\delta_{c_\square\equiv i}(-\ddd)^{d_{i+1}(\lambda)}\delta\left( \frac{ z}{\chi_\square\upsilon}\right)
\frac{\displaystyle\prod_{\blacksquare\in R_{i}(\lambda)}\left( \chi_\square-\qqq^2\chi_\blacksquare \right)}
{\displaystyle\prod_{\substack{\blacksquare\in A_{i}(\lambda)\\\blacksquare\not=\square}}\left(\chi_\square-\chi_\blacksquare\right)}\\
\langle\lambda+\square |\tau^-_{\upsilon}(f_i(z))|\lambda\rangle&=\delta_{c_\square\equiv i}(-\ddd)^{-d_{i+1}(\lambda)}\delta\left(\frac{ z}{\chi_\square\upsilon}\right)
\frac{\displaystyle\prod_{\substack{\blacksquare\in A_{i}(\lambda)\\\blacksquare\not=\square}}\left( \qqq\chi_\square-\qqq^{-1}\chi_\blacksquare \right)}
{\displaystyle\prod_{\blacksquare\in R_{i}(\lambda)}\qqq\left( \chi_\square-\chi_\blacksquare \right)}\\
\langle\lambda|\tau^-_{\upsilon}(\psi_i^\pm(z))|\lambda\rangle&=
\prod_{\blacksquare\in A_{i}(\lambda)}\frac{\left(\qqq z-\qqq^{-1}\chi_{\blacksquare}\upsilon\right)}{\left( z-\chi_\blacksquare\upsilon\right)}
\prod_{\blacksquare\in R_{i}(\lambda)}\frac{\left(\qqq^{-1} z-\qqq\chi_\blacksquare\upsilon\right)}{\left( z-\chi_\blacksquare\upsilon\right)},
\end{aligned}\\
\langle\lambda|\tau_{\upsilon}^-(\gamma^{\frac{1}{2}})|\lambda\rangle=1,\,\langle\lambda|\tau_{\upsilon}^-(\qqq^{d_2})|\lambda\rangle=\qqq^{-|\lambda|}
\end{gather*}
On the other hand, if we identify
\[
q=\qqq^{-1}\ddd,\, t=\qqq^{-1}\ddd^{-1}
\]
then we can define another $'\ddot{U}$-action $\tau_{\upsilon}^+$ on $\mathcal{F}$ with matrix elements
\begin{gather*}
\begin{aligned}
\langle\lambda +\square| \tau_{\upsilon}^+(e_i(z))|\lambda\rangle&=\delta_{c_\square\equiv i}(-\ddd)^{-d_{i+1}(\lambda)}\delta\left(\frac{z}{\chi_\square \upsilon}\right)
\frac{\displaystyle\prod_{\substack{\blacksquare\in A_{i}(\lambda)\\\blacksquare\not=\square}}\left(\chi_\square-\qqq^2\chi_\blacksquare\right)}
{\displaystyle\prod_{\blacksquare\in R_{i}(\lambda)}\left( \chi_\square-\chi_\blacksquare \right)}\\
\langle\lambda |\tau_{\upsilon}^+(f_i(z))|\lambda+\square\rangle&=\delta_{c_\square\equiv i}(-\ddd)^{d_{i+1}(\lambda)}\delta\left(\frac{z}{\chi_\square \upsilon}\right)
\frac{\displaystyle\prod_{\blacksquare\in R_{i}(\lambda)}\left( \qqq\chi_\square-\qqq^{-1}\chi_\blacksquare \right)}
{\displaystyle\prod_{\substack{\blacksquare\in A_{i}(\lambda)\\\blacksquare\not=\square}}\qqq\left( \chi_\square-\chi_\blacksquare \right)}\\
\langle\lambda|\tau_{\upsilon}^+(\psi_i^\pm(z))|\lambda\rangle&=
\prod_{\blacksquare\in A_{i}(\lambda)}\frac{\left(\qqq^{-1}z-\qqq\chi_\blacksquare \upsilon\right)}{\left(z-\chi_\blacksquare \upsilon\right)}
\prod_{\blacksquare\in R_{i}(\lambda)}\frac{\left(\qqq z-\qqq^{-1}\chi_\blacksquare \upsilon\right)}{\left(z-\chi_\blacksquare \upsilon\right)},
\end{aligned}\\
\langle\lambda|\tau_{\upsilon}^+(\gamma^{\frac{1}{2}})|\lambda\rangle=1,\,\langle\lambda|\tau_{\upsilon}^+(\qqq^{d_2})|\lambda\rangle=\qqq^{|\lambda|}
\end{gather*}
Both representations are irreducible.
\end{prop}
\noindent We call $\tau_\upsilon^-$ the \textit{highest weight} Fock representation and $\tau_\upsilon^+$ the \textit{lowest weight} Fock representation.

\begin{rem}
We have presented $\tau_\upsilon^\pm$ in a way that hides its fermionic nature.
Following ideas from \cite{NagaoK}, one can show that the both representations are isomorphic to representations on the $\qqq$-deformed Fock space of \cite{KMS}.
Our basis here corresponds to a `spinon' basis rather than to $\qqq$-deformed semiinfinite wedges (cf. \cite{JKKMP}).
\end{rem}

\subsubsection{Tsymbaliuk isomorphisms}
The relationship between the vertex and Fock representations was established by Tsymbaliuk \cite{Tsym} (cf. 3.5.2 of \cite{WreathEigen}):
\begin{thm}\label{Eigenstates}
The vacuum-to-vacuum map
\[\mathbb{1}\otimes1 \mapsto |\varnothing\rangle\]
induces an isomorphism $\mathrm{T}_-:W\rightarrow\mathcal{F}$ between the $'\ddot{U}$-modules $\rho_{\vec{\ccc}}\circ\varpi$ and $\tau_{\upsilon}^-$, where the parameters $\vec{c}$ and $\upsilon$ are related by
\[\upsilon=(-1)^{\frac{(\ell-1)(\ell-2)}{2}}\frac{\qqq\ddd^{-\frac{\ell}{2}}}{\ccc_0\cdots \ccc_{\ell-1}}\]
It also induces an isomorphism $\mathrm{T}_+:W\rightarrow\mathcal{F}$ between the $'\ddot{U}'$-modules $\rho_{\vec{\ccc}}\circ\varpi^{-1}$ and $\tau_{\upsilon}^+$, where the parameters are instead related by
\[
\upsilon=(-1)^{\frac{(\ell-1)(\ell-2)}{2}}\qqq^{-1}\ddd^{-\frac{\ell}{2}}\ccc_0\cdots\ccc_{\ell-1}
\]
\end{thm}

We note that while both isomorphisms are induced from the vacuum-to-vacuum map, they are propagated to the rest of $W$ using different actions of the quantum toroidal algebra.
Our previous result \cite{WreathEigen} adds more detail to $\mathrm{T}_-$.
Recall that we have identfied $W$ and $\Lambda_{q,t}^{\otimes\ell}\otimes\FF[Q]$ using (\ref{PowertoBosons}), and thus we can view the wreath Macdonald polynomials $H_\lambda\otimes e^{\core(\lambda)}$ as elements of $W$.
\begin{thm}\label{WreathEigen}
The Tsymabliuk isomorphisms $\mathrm{T}_\pm$ satisfy the following:
\begin{enumerate}
\item Under the identification (\ref{PowertoBosons}) and the matching of parameters
\begin{equation}
q=\qqq\ddd,\, t=\qqq\ddd^{-1}
\label{PlusPar}
\end{equation}
the Tsymbaliuk isomorphism $\mathrm{T}_-$ sends $\FF (H_\lambda\otimes e^{\core(\lambda)})$ to $\FF|\lambda\rangle$.
\item On the other hand, under the identification (\ref{PowertoBosons}) and the matching of parameters
\[
q=\qqq^{-1}\ddd,\, t=\qqq^{-1}\ddd^{-1}
\]
the Tsymbaliuk isomorphism $\mathrm{T}_+$ sends $\FF (H_\lambda\left[ -X^\bullet \right]\otimes e^{\core(\lambda)})$ to $\FF|\lambda\rangle$.
\end{enumerate}
\end{thm}
Part (1) was proved in \cite{WreathEigen}.
The proof of part (2) is similar once we have established Propositions \ref{PosEnHn} and \ref{Adjacency} below.
We have emphasized that the matching of toroidal parameters $(\qqq,\ddd)$ and symmetric function parameters $(q,t)$ differ between the two isomorphisms $\mathrm{T}_\pm$.
We will take the approach of adhering to the different parameter matchings when making computations and then reverting to (\ref{PlusPar}) when stating final results about wreath Macdonald polynomials.
This amounts to performing the swap $q\leftrightarrow t^{-1}$.
Doing this makes Theorem \ref{WreathEigen}(2) really a statement about $H_\lambda^*$.
We highlight here that in $W$,
\begin{itemize}
\item the basis $\left\{ H_\lambda\otimes e^{\core(\lambda)} \right\}$ diagonalizes $\varpi(\Heis)$;
\item the basis $\left\{ H_\lambda^*\otimes e^{\core(\lambda)} \right\}$ diagonalizes $\varpi^{-1}(\Heis)$.
\end{itemize} 

\subsection{Shuffle algebra}
Utilizing the Tsymbaliuk isomorphisms, we can transfer the study of $W$ to that of $\mathcal{F}$ after applying $\varpi^{\pm 1}$.
However, studying the images of elements of $\UTor$ under Miki's automorphism in terms of generators and relations is extremely difficult.
To access these images and their actions, we will utilize the \textit{shuffle algebra}. 

\subsubsection{Definition}
Let $\vec{k}=(k_0,\ldots, k_{\ell-1})\in(\ZZ_{\ge 0})^{\ZZ/\ell\ZZ}$.
The space of rational functions
\[\FF(x_{i,r})_{i\in\ZZ/\ell\ZZ,1\le r\le k_i}\]
carries an action of the group
\[\Sigma_{\vec{k}}:=\prod_i\Sigma_{k_i}\]
wherein the factor $\Sigma_{k_i}$ only permutes the variables $\{x_{i,r}\}_{1\le r\le k_i}$.
For the variable $x_{i,r}$, we call $i$ its \textit{color}, so $\Sigma_{\vec{k}}$ acts via \textit{color-preserving} permuations.
We will be concerned with the space of \textit{color-symmetric} rational functions:
\begin{align*}
\mathbb{S}_{\vec{k}}&:=\FF(x_{i,r})_{i\in\ZZ/\ell\ZZ,1\le r\le k_i}^{\Sigma_{\vec{k}}},&\mathbb{S}&:=\bigoplus_{\vec{k}\in(\ZZ_{\ge0})^{\ZZ/\ell\ZZ}}\mathbb{S}_{\vec{k}}
\end{align*}
Finally, we will need the notation
\begin{align*}
|\vec{k}|&:=\sum_{i\in\ZZ/\ell\ZZ}k_i\\
\vec{k}!&:=\prod_{i\in\ZZ/\ell\ZZ}k_i!
\end{align*}

On $\mathbb{S}$, we can define the \textit{shuffle product} $\star$.
First, we need the \textit{mixing terms}:
\[\omega_{i,j}(z,w):=\left\{\begin{array}{ll}
\left(z-\qqq^{2}w\right)^{-1}\left(z-w\right)^{-1} & \hbox{if }i=j\\
\left(\qqq w-\ddd^{-1}z\right) &\hbox{if }i+1=j\\
\left(z-\qqq\ddd^{-1} w\right) &\hbox{if }i-1=j\\
1 &\hbox{otherwise}
\end{array}\right.\]
For $F\in\mathbb{S}_{\vec{n}}$ and $G\in\mathbb{S}_{\vec{m}}$, $F\star G\in\mathbb{S}_{\vec{n}+\vec{m}}$ is given by
\[\frac{1}{\vec{n}!\vec{m}!}\Sym\left(F(\{x_{i,r}\}_{r\le n_i})G(\{x_{j,s}\}_{n_j<s})\prod_{i,j\in\ZZ/\ell\ZZ}\prod_{\substack{r\le n_i\\ n_j<s}}\omega_{i,j}(x_{i,r}/x_{j,s})\right)\]
where $\Sym$ denotes the \textit{color symmetrization}: for $f\in\FF(\{x_{i,1},\ldots,x_{i,k_i}\}_{i\in\ZZ/\ell\ZZ})$,
\[\Sym(f):=\sum_{(\sigma_0,\ldots,\sigma_{\ell-1})\in\prod\Sigma_{\vec{k}}}f(\{x_{i,\sigma_i(r)}\})\]

Let $\Sss_{\vec{k}}\subset\mathbb{S}_{\vec{k}}$ be the subspace of functions $F$ satisfying:
\begin{enumerate}
\item \textit{Pole conditions:} $F$ is of the form
\begin{equation*}
F=\frac{f(\{x_{i,r}\})}{\prod_{i\in\ZZ/\ell\ZZ}\prod_{r\not= r'}(x_{i,r}-\qqq^2x_{i,r'})}
\end{equation*}
for a color-symmetric Laurent polynomial $f$.
\item \textit{Wheel conditions:} $F$ has a well-defined finite limit when
\[\frac{x_{i,r_1}}{x_{i+\epsilon,s}}\rightarrow\qqq\ddd^\epsilon\hbox{ and }\frac{x_{i+\epsilon,s}}{x_{i,r_2}}\rightarrow\qqq\ddd^{-\epsilon}\]
for any choice of $i$, $r_1$, $r_2$, $s$, and $\epsilon$, where $\epsilon\in\{\pm 1\}$. 
This is equivalent to specifying that the Laurent polynomial $f$ in the pole conditions evaluates to zero.
\end{enumerate}
We set
\[\Sss:=\bigoplus_{\vec{k}\in(\ZZ_{\ge0})^{\ZZ/\ell\ZZ}}\Sss_{\vec{k}}\]
\begin{prop}
The product $\star$ is associative and $\Sss$ is closed under it.
\end{prop}
\noindent $\Sss$ is called the \textit{shuffle algebra of type $\hat{A}_{\ell-1}$}. 
Let $\Sss^+:=\Sss$ and $\Sss^-:=\Sss^{\mathrm{opp}}$.

\subsubsection{Relation to $\UTor$}
Let $\ddot{U}^+, \ddot{U}^-\subset\UTor$ be the subalgebras generated by 
\[\left\{ e_{i,n} \right\}_{i\in\ZZ/\ell\ZZ}^{n\in\ZZ}\hbox{ and }\left\{ f_{i,n} \right\}_{i\in\ZZ/\ell\ZZ}^{n\in\ZZ}\]
respectively.
The following fundamental result was proved by Negu\cb{t}:
\begin{thm}[\cite{NegutTor}]
There exist algebra isomorphisms $\Psi_+:\Sss^+\rightarrow\ddot{U}^+$ and $\Psi_-:\Sss^-\rightarrow\ddot{U}^-$ such that
\begin{align*}
\Psi_+(x_{i,1}^k)&=e_{i,k}\\
\Psi_-(x_{i,1}^k)&=f_{i,k}
\end{align*}
\end{thm}
\noindent Therefore, we can use $\Sss$ to give alternative presentations of elements of $\ddot{U}^\pm$.
The utility of this is bolstered by the following:

\begin{prop}[\cite{NegutCyc}, cf. Proposition 4.8 of \cite{WreathEigen}]\label{ShuffleFock}
For $F\in \Sss^+_{\vec{k}}$ and $G\in\Sss^-_{\vec{k}}$, the actions of $\Psi_{+}(F)$ and $\Psi_-(G)$ on $\mathcal{F}$ via $\tau_\upsilon^\pm$ are such that the only nonzero matrix elements involve pairs of partitions $\mu$ and $\lambda$ where $\mu$ adds $k_i$ $i$-nodes to $\lambda$ for all $i\in\ZZ/\ell\ZZ$. 
In this case, assign to each node $\square\in\mu\backslash\lambda$ its own variable $x_\square$ of color $\bar{c}_\square$. 
We then have
\begin{align}
\label{LowElement}
\langle\mu|(\tau_{\upsilon}^+\circ\Psi_+)(F)|\lambda\rangle 
&= \left( F(\{x_{i,r}\}) \bigg|_{x_\square\rightarrow\chi_\square\upsilon}\right)
\prod_{\square\in\mu\backslash\lambda}\frac{\left( \chi_\square-\qqq^2 \right)^{\delta_{c_\square\equiv 0}}}{\chi_\square\left( 1-\qqq^{2} \right)}
\prod_{\blacksquare\in\lambda}\omega_{\bar{c}_{\square},\bar{c}_{\blacksquare}}\left( \chi_{\square},\chi_{\blacksquare} \right)\\
\label{HighElement}
\langle\mu|(\tau_{\upsilon}^-\circ\Psi_-)(G)|\lambda\rangle
&= \left(G(\{x_{i,r}\})\bigg|_{x_{\square}\rightarrow \chi_{\square}\upsilon}\right)
\prod_{\square\in\mu\backslash\lambda}
\frac{\left( \qqq\chi_\square-\qqq^{-1} \right)^{\delta_{c_{\square}\equiv 0}}}
{\chi_\square(\qqq-\qqq^{-1})}
\displaystyle\prod_{\blacksquare\in\lambda}\omega_{\bar{c}_{\blacksquare},\bar{c}_\square}(\chi_{\blacksquare},\chi_\square)
\end{align}
\end{prop}
\noindent Formula (\ref{HighElement}) was proven in detail in \cite{WreathEigen}.
The proof of (\ref{LowElement}) follows the same scheme, although we emphasize that the matching of parameters $(\qqq,\ddd)$ with $(q,t)$ differ in the two Fock representations.

\begin{rem}
Note that the formulas (\ref{LowElement}) and (\ref{HighElement}) are only for box \textit{additions}.
Analogous formulas for box removals exist---the case of $\tau_\upsilon^-$ was worked out in \cite{WreathEigen}.
\end{rem}

\subsubsection{Pieri operators}
Wreath analogues of Pieri and dual Pieri rules for Macdonald polynomials would entail writing the products
\[
\textstyle e_n\left[ X^{(i)} \right]P_\lambda
\hbox{ and }
\displaystyle h_n\left[ \frac{(1-t\sigma^{-1})}{(1-q\sigma^{-1})}X^{(i)} \right] P_\lambda
\]
in terms of the ordinary (wreath) Macdonald basis $\{P_\mu\}$.
Our chosen transformation for $h_n$ is the wreath analogue of the transformation sending $h_n$ to the $g_n$ of \cite{Mac}.
Inverting $t$ and performing a plethysm to obtain scalar multiples of $\{H_\mu\}$, we can equivalently study 
\begin{align*}
&\textstyle e_n\left[(1-t^{-1}\sigma^{-1})^{-1} X^{(i)} \right]\tilde{P}_\lambda(q,t^{-1} )\\
\hbox{and }&\textstyle h_n\left[ (1-q\sigma^{-1})^{-1}X^{(i)} \right] \tilde{P}_\lambda(q,t^{-1}) 
\end{align*}
If we match parameters using 
\[
q=\qqq\ddd,\, t=\qqq\ddd^{-1}
\]
then by Theorem \ref{WreathEigen}, we can use $\mathrm{T}_-$ to carry this study over to $\mathcal{F}$.
Using the identification (\ref{PowertoBosons}), we would be interested in studying the action on $\mathcal{F}$ of
\begin{align*}
&\textstyle (\tau_\upsilon^-\circ\varpi^{-1})\left(e_n\left[(1-t^{-1}\sigma^{-1})^{-1} X^{(i)} \right]\right)\\
\hbox{and }&\textstyle (\tau_\upsilon^-\circ\varpi^{-1})\left(h_n\left[ (1-q\sigma^{-1})^{-1}X^{(i)} \right]\right) 
\end{align*}
We can then use (\ref{HighElement}) provided that these elements (without $\tau_\upsilon^-$) are in the image of $\Psi_-$ and we have formulas for their preimages.

To that end, let
\begin{align}
\label{Epn}
E_{p,n}&:=
\Sym\left( \prod_{1\le r<s\le n}\left\{\frac{x_{p+1,r}-\qqq^{-1}\ddd^{-1}x_{p,s}}{x_{p+1,r}-\qqq\ddd^{-1}x_{p,s}}\prod_{i,j\in\ZZ/\ell\ZZ}\omega_{i,j}\left( x_{i,r},x_{j,s} \right)\right\}\right.\\
&\times \left.\prod_{r=1}^n\left\{\left( \qqq^{-1}\ddd^{-1}\frac{x_{0,r}}{x_{p+1,r}}-\frac{x_{0,r}}{x_{p,r}} \right)\prod_{i\in\ZZ/\ell\ZZ}x_{i,r}  \right\} \right)\\
\label{Hpn}
H_{p,n}&:=
\Sym\left( \prod_{1\le r<s\le n}\left\{\frac{\qqq^{-1}\ddd x_{p+1,s}-x_{p,r}}{\qqq\ddd x_{p+1,s}-x_{p,r}}\prod_{i,j\in\ZZ/\ell\ZZ}\omega_{i,j}\left( x_{i,r},x_{j,s} \right)\right\}\right.\\
&\times \left.\prod_{r=1}^n\left\{\left( \qqq\ddd^{-1}\frac{x_{0,r}}{x_{p+1,r}}-\frac{x_{0,r}}{x_{p,r}} \right)\prod_{i\in\ZZ/\ell\ZZ}x_{i,r}  \right\} \right)
\end{align}
The following is Theorem 4.31 of \cite{WreathEigen}:
\begin{prop}\label{ShufflePres}
The functions $E_{p,n}$ and $H_{p,n}$ lie in $\Sss$.
Under the matching of parameters
\[
q=\qqq\ddd,\, t=\qqq\ddd^{-1}
\]
we have
\begin{align*}
\varpi^{-1}e_n\left[(1-t^{-1}\sigma^{-1})^{-1} X^{(p)} \right]&=\frac{(1-qt)^{n\ell}}{\prod_{r=1}^n(1-(qt)^{-r})}\Psi_- (E_{p,n})\\
\varpi^{-1}h_n\left[ (1-q\sigma^{-1})^{-1}X^{(p)} \right]&= \frac{(qt)^{-n}(1-qt)^{n\ell}}{\prod_{r=1}^n(1-(qt)^{-r})}\Psi_-(H_{p,n})
\end{align*}
\end{prop}

Alternatively, we can use the matching
\[
q=\qqq^{-1}\ddd,\, t=\qqq^{-1}\ddd^{-1}
\]
and use Theorem \ref{WreathEigen}(2).
To carry this out, we would use plethysms to transfer the Pieri rules over to 
\begin{align*}
&\textstyle e_n\left[(1-t^{-1}\sigma^{-1})^{-1} (-X^{(i)}) \right]\tilde{P}_\lambda\left[ -X^\bullet;q,t^{-1} \right]\\
\hbox{and }&\textstyle h_n\left[ (1-q\sigma^{-1})^{-1}(-X^{(i)}) \right] \tilde{P}_\lambda\left[ -X^\bullet;q,t^{-1} \right]
\end{align*}
We would then be interested in the action on $\mathcal{F}$ of
\begin{align*}
&\textstyle (\tau_\upsilon^+\circ\varpi)\left( e_n\left[(1-t^{-1}\sigma^{-1})^{-1} (-X^{(i)}) \right]\right)\\
\hbox{and }&\textstyle (\tau_\upsilon^+\circ\varpi)\left( h_n\left[ (1-q\sigma^{-1})^{-1}(-X^{(i)} )\right]\right) 
\end{align*}
\begin{prop}\label{PosEnHn}
Under the matching of parameters
\[
q=\qqq^{-1}\ddd,\, t=\qqq^{-1}\ddd^{-1}
\]
we have
\begin{align}
\label{PosEn}
\textstyle \varpi e_n\left[(1-t^{-1}\sigma^{-1})^{-1} (-X^{(p)}) \right]&=\frac{(-1)^n(qt)^n(\qqq-\qqq^{-1})^{n\ell}}{\prod_{r=1}^n(1-(qt)^r)}\Psi_+(H_{p,n})\\ 
\label{PosHn}
\textstyle \varpi h_n\left[ (1-q\sigma^{-1})^{-1}(-X^{(p)} )\right]&= \frac{(-1)^n(\qqq-\qqq^{-1})^{n\ell}}{\prod_{r=1}^n(1- (qt)^r)}\Psi_+(E_{p,n})
\end{align}
\end{prop}

\begin{proof}
First consider (\ref{PosEn}). 
Our goal is to utilize Proposition B.1 of \cite{WreathEigen}.
Although not explicitly mentioned in the proposition, we are using (\ref{PowertoBosons}):
\begin{align}
\label{PosEnHeis}
\sum_{n=0}^\infty\textstyle\varpi e_n\left[(1-t^{-1}\sigma^{-1})^{-1} (-X^{(p)}) \right] (-z)^n&=\varpi\exp\left( \sum_{k=1}^\infty\frac{\sum_{i=0}^{\ell-1}t^{-ki}b_{p+i,-k}}{[k]_\qqq(1-t^{-k\ell})}z^k \right)
\end{align}
By Theorem \ref{MikiAut}, we have $\varpi=\eta\varpi^{-1}\eta$.
Applying this to (\ref{PosEnHeis}) yields
\begin{align*}
&= \eta\exp\left( -\sum_{k=1}^\infty\frac{\sum_{i=0}^{\ell-1}q^{-ki}\varpi^{-1}(b_{p+i,k})}{[k]_\qqq(1-q^{-k\ell})}z^k \right)
\end{align*}
Through computations similar to those of Lemma 3.15 of \cite{WreathEigen}, we can see that
\[
\frac{-\sum_{i=0}^{\ell-1}q^{-ki}\varpi^{-1}(b_{p+i,k})}{[k]_\qqq(1-q^{-k\ell})}=\frac{\qqq^{-k}\varpi^{-1}(b_{p,k}^\perp)-\ddd^k \varpi^{-1}(b_{p+1,k}^\perp)}{{\qqq-\qqq^{-1}}}
\]
Here, $b_{i,k}^\perp$ is defined in 3.1.3 of \textit{loc. cit.} in terms of a pairing.
Note that Proposition B.1 of \textit{loc. cit.} is written in terms of the other matching of parameters
\[
q=\qqq\ddd,\,t=\qqq\ddd^{-1}
\]
Applying it correctly here gives us
\begin{equation}
\sum_{n=0}^\infty\textstyle\varpi e_n\left[(1-t^{-1}\sigma^{-1})^{-1} (-X^{(p)}) \right] z^n
\displaystyle=
\eta\left(\sum_{n=0}^\infty \frac{(-1)^{n\ell}q^{-n\ell}\left( 1-\qqq^{-2} \right)^{n\ell}}{\qqq^{2n}\prod_{r=1}^n(1-(qt)^r)}\Psi_+(E_{p,n}^\dagger)\right)
\label{PosEnHeisInv}
\end{equation}
where
\begin{align*}
E_{p,n}^\dagger
&=\Sym\left( \prod_{1\le r<s\le n}\left\{\frac{x_{p+1,r}-\qqq^{-1}\ddd^{-1}x_{p,s}}{x_{p+1,r}-\qqq\ddd^{-1}x_{p,s}}\prod_{i,j\in\ZZ/\ell\ZZ}\omega_{i,j}\left( x_{i,r},x_{j,s} \right)\right\}\right.\\
&\times \left.\prod_{r=1}^n\left\{\left( \frac{x_{p,r}}{x_{0,r}}-\qqq\ddd\frac{x_{p+1,r}}{x_{0,r}} \right)\prod_{i\in\ZZ/\ell\ZZ}x_{i,r}  \right\} \right)
\end{align*}
Finally, we observe that $\eta$ induces the following map on $\Sss^+$:
\[
\Psi_+^{-1}\eta\Psi_+(F)
=\left.F(x_{i,r}^{-1})\prod_{i\in\ZZ\ell\ZZ}\prod_{r=1}^{k_i}(-\ddd)^{k_{i+1}k_i} x_{i,r}^{k_{i+1}+k_{i-1}-2(k_i-1)}\right|_{\ddd\mapsto\ddd^{-1}}
\]
Using this to deal with $\eta$, the right hand side of (\ref{PosEnHeisInv}) becomes
\begin{align*}
&= \sum_{n=0}^\infty\frac{t^{-n\ell}(1-\qqq^{-2})^{n\ell}}{\qqq^{2n}\prod_{r=1}^n(1-(qt)^r)}(-1)^{n}\ddd^{-n\ell}\Psi_+(H_{p,n})
\end{align*}
from which (\ref{PosEn}) follows.
The proof of (\ref{PosHn}) is similar.
\end{proof}

We end by applying Proposition \ref{ShuffleFock} to $E_{p,n}$ and $H_{p,n}$.
The zeroes of $E_{p,n}$ and $H_{p,n}$ will force certain matrix elements to vanish.
Specifically, we obtain:

\begin{prop}\label{Adjacency}
The matrix elements 
\[\langle\mu|(\tau_{\upsilon}^-\circ\Psi_-)(E_{p,n})|\lambda\rangle\hbox{ and }\langle\mu|(\tau_\upsilon^+\circ\Psi_+)(H_{p,n})|\lambda\rangle\]
are nonzero only if $\mu\backslash\lambda$ has exactly $n$ nodes of each color and no horizontally adjacent $p$- and $(p+1)$-nodes.
Similarly, 
\[\langle\mu|(\tau_\upsilon^-\circ\Psi_-)(H_{p,n})|\lambda\rangle\hbox{ and }\langle\mu|(\tau_\upsilon^+\circ\Psi_+)(E_{p,n})|\lambda\rangle\]
are nonzero only if $\mu\backslash\lambda$ has exactly $n$ nodes of each color and no vertically adjacent $p$- and $(p+1)$-nodes.
\end{prop}

\section{Normalization}\label{Normal}
Here, we highlight a normalization quantity that directly yields the inner products $\langle P_{\tensor[^t]{\lambda}{}}^*, P_\lambda\rangle_{q,t}$.
Computing this normalization is also key for applying the shuffle algebra to studying Pieri rules.
\subsection{Setup}
When
\[
q=\qqq^{-1}\ddd,\, t=\qqq^{-1}\ddd^{-1}
\]
we define the scalars $N_\lambda^+(q,t)$ and $M_\lambda^+(q,t)$ by
\begin{align*}
(\tau_\upsilon^+\circ\varpi)\left(\tilde{P}_\lambda[-X^\bullet;q,t^{-1}]  \right)|\core(\lambda)\rangle
&= N_\lambda^+(q,t^{-1})|\lambda\rangle\\
(\tau_\upsilon^+\circ\varpi)\left(\tilde{Q}_\lambda[-X^\bullet;q,t^{-1}]  \right)|\core(\lambda)\rangle
&= M_\lambda^+(q,t^{-1})|\lambda\rangle
\end{align*}
Similarly, with
\[
q=\qqq\ddd,\, t=\qqq\ddd^{-1}
\]
we define $N_\lambda^-(q,t)$ and $M^-_\lambda(q,t)$ as
\begin{align*}
(\tau_\upsilon^-\circ\varpi^{-1})\left(\tilde{P}_\lambda(q,t^{-1})  \right)|\core(\lambda)\rangle
&= N_\lambda^-(q,t^{-1})|\lambda\rangle\\
(\tau_\upsilon^-\circ\varpi^{-1})\left(\tilde{Q}_\lambda(q,t^{-1})  \right)|\core(\lambda)\rangle
&= M_\lambda^-(q,t^{-1})|\lambda\rangle
\end{align*}
Our interest in these scalars comes from the following trivial observation:
\begin{prop}
The norm formulas can be computed as follows:
\begin{align*}
\langle P^*_{\tensor[^t]{\lambda}{}},P_\lambda\rangle_{q,t}&= \frac{N^+_{\tensor[^t]{\lambda}{}}(t,q)}{M^+_{\tensor[^t]{\lambda}{}}(t,q)}=\frac{N^-_\lambda(q,t)}{M^-_\lambda(q,t)}\\
\langle Q^*_{\tensor[^t]{\lambda}{}},Q_\lambda\rangle_{q,t}&= \frac{M^+_{\tensor[^t]{\lambda}{}}(t,q)}{N^+_{\tensor[^t]{\lambda}{}}(t,q)}=\frac{M^-_\lambda(q,t)}{N^-_\lambda(q,t)}
\end{align*}
\end{prop}

\begin{proof}
The key observation is:
\[
P_\lambda=\frac{N_\lambda^-(q,t)}{M_\lambda^-(q,t)} Q_\lambda
\]
This yields:
\[
\langle P^*_{\tensor[^t]{\lambda}{}},P_\lambda\rangle_{q,t}=\frac{N_\lambda^-(q,t)}{M_\lambda^-(q,t)}\langle P^*_{\tensor[^t]{\lambda}{}},Q_\lambda\rangle_{q,t}=\frac{N_\lambda^-(q,t)}{M_\lambda^-(q,t)}\cdot 1
\]
The proof of the other equalities is similar.
\end{proof}

\subsection{Pieri triangularity}
Recall that in the case $\ell=1$, the $\{s_\lambda\}$, $\{e_\lambda\}$ and $\{h_\lambda\}$ satisfy the following triangularity conditions with respect to dominance order (cf. \cite{Mac} I.6):
\begin{align*}
s_\lambda&= e_\lambda+\sum_{\mu< \lambda}\alpha_{\lambda\mu}e_\mu\\
&= h_\lambda+\sum_{\mu>\lambda}\alpha'_{\lambda\mu}h_\mu
\end{align*}
for some $\alpha_{\lambda\mu},\alpha'_{\lambda\mu}\in\ZZ$ (recall our convention (\ref{ElemConv}) for $e_\lambda$).
It is not difficult to see that for general $\ell$, this implies (cf. \cite{WreathEigen} Remark 2.7(3)):
\begin{align*}
s_{\quot(\lambda)}&= e_{\quot(\lambda)}+\sum_{\mu<_\ell \lambda}\beta_{\lambda\mu}e_{\quot(\mu)}\\
&= h_{\quot(\lambda)}+\sum_{\mu>_\ell\lambda}\beta'_{\lambda\mu}h_{\quot(\mu)}
\end{align*}
for some $\beta_{\lambda\mu},\beta_{\lambda\mu}'\in\ZZ$.
Consequently,
\begin{align*}
\tilde{P}_\lambda(q,t^{-1} )&= e_{\quot(\lambda)}\left[ (1-t^{-1}\sigma^{-1})X^\bullet \right]
+\sum_{\mu<_\ell\lambda}\gamma_{\lambda\mu}(q,t)e_{\quot(\mu)}\left[ (1-t^{-1}\sigma^{-1})X^\bullet \right]\\
\tilde{Q}_\lambda(q,t^{-1} )&= h_{\quot(\lambda)}\left[ (1-q\sigma^{-1})X^\bullet \right]
+\sum_{\mu>_\ell\lambda}\gamma_{\lambda\mu}'(q,t)h_{\quot(\mu)}\left[ (1-q\sigma^{-1})X^\bullet \right]
\end{align*}
for some $\gamma_{\lambda\mu}(q,t),\gamma_{\lambda\mu}'(q,t)\in\CC(q,t)$.
Since the inverse of a unitriangular matrix is also unitriangular, we obtain
\begin{equation}
\begin{aligned}
e_{\quot(\lambda)}\left[ (1-t^{-1}\sigma^{-1})X^\bullet \right]&= \tilde{P}_\lambda(q,t^{-1} )
+\sum_{\mu<_\ell\lambda}\delta_{\lambda\mu}(q,t)\tilde{P}_\mu(q,t^{-1} )\\
h_{\quot(\lambda)}\left[ (1-q\sigma^{-1})X^\bullet \right]&= \tilde{Q}_\lambda(q,t^{-1} )
+\sum_{\mu>_\ell\lambda}\delta'_{\lambda\mu}(q,t)\tilde{Q}_\mu(q,t^{-1} )
\end{aligned}
\label{EHtoPQ}
\end{equation}
for some $\delta_{\lambda\mu}(q,t),\delta_{\lambda\mu}'(q,t)\in\CC(q,t)$.
Thus,
\begin{equation}
\begin{aligned}
N^+_\lambda(q,t^{-1})&= \left\langle\lambda\left|(\tau^+_\upsilon\circ\varpi)\left(\tilde{P}_\lambda\left[ -X^\bullet ;q,t^{-1} \right]  \right)\right|\core(\lambda)\right\rangle\\
&= \left\langle\lambda\left|(\tau^+_\upsilon\circ\varpi)\left(e_{\quot(\lambda)}\left[ (1-t^{-1}\sigma^{-1})^{-1}(-X^\bullet)  \right]  \right)\right|\core(\lambda)\right\rangle\\
M^+_\lambda(q,t^{-1})&= \left\langle\lambda\left|(\tau^+_\upsilon\circ\varpi)\left(\tilde{Q}_\lambda\left[ -X^\bullet ;q,t^{-1} \right]  \right)\right|\core(\lambda)\right\rangle\\
&= \left\langle\lambda\left|(\tau^+_\upsilon\circ\varpi)\left(h_{\quot(\lambda)}\left[ (1-q\sigma^{-1})^{-1}(-X^\bullet)  \right]  \right)\right|\core(\lambda)\right\rangle\\
N^-_\lambda(q,t^{-1})&= \left\langle\lambda\left|(\tau_\upsilon^-\circ\varpi^{-1})\left(\tilde{P}_\lambda(q,t^{-1} )  \right)\right|\core(\lambda)\right\rangle\\
&= \left\langle\lambda\left|(\tau^-_\upsilon\circ\varpi^{-1})\left(e_{\quot(\lambda)}\left[ (1-t^{-1}\sigma^{-1})^{-1}X^\bullet  \right]  \right)\right|\core(\lambda)\right\rangle\\
M^-_\lambda(q,t^{-1})&= \left\langle\lambda\left|(\tau^-_\upsilon\circ\varpi^{-1})\left(\tilde{Q}_\lambda(q,t^{-1}) \right)\right|\core(\lambda)\right\rangle\\
&= \left\langle\lambda\left|(\tau^-_\upsilon\circ\varpi^{-1})\left(h_{\quot(\lambda)}\left[ (1-q\sigma^{-1})^{-1}X^\bullet  \right]  \right)\right|\core(\lambda)\right\rangle
\end{aligned}
\label{MatrixEl}
\end{equation}

The equations (\ref{MatrixEl}) give a pathway to using our shuffle elements $E_{p,n}$ and $H_{p,n}$ to compute $N_\lambda^\pm(q,t)$ and $M_\lambda^\pm(q,t)$.
We would like to compute these matrix elements by an inductive procedure along the rows or columns of $\quot(\lambda)$.
The combinatorics of doing so is subtle, and this is where the orders reviewed in \ref{Orders} come in.

\begin{lem}\label{Induct}
For a partition $\lambda$, suppose the final column in the left-to-right order on columns of $\quot(\lambda)$ has length $n$ and lies in component $p$.
Let $\lambda'\subset\lambda$ be the partition obtained by removing this column.
Then
\begin{align*}
N_\lambda^+(q,t^{-1})&=N_{\lambda'}^+(q,t^{-1})\left\langle\lambda\left|(\tau_\upsilon^+\circ\varpi)\left(e_n\left[(1-t^{-1}\sigma^{-1})^{-1}(-X^{(p)}) \right]\right)\right|\lambda'\right\rangle\\
N_\lambda^-(q,t^{-1})&=N_{\lambda'}^-(q,t^{-1})\left\langle\lambda\left|(\tau_\upsilon^-\circ\varpi^{-1})\left(e_n\left[(1-t^{-1}\sigma^{-1})^{-1}X^{(p)} \right]\right)\right|\lambda'\right\rangle
\end{align*}
On the other hand, suppose the final row in the right-to-left order on rows of $\quot(\lambda)$ has length $n^*$ and lies in component $p^*$.
Letting $\lambda''\subset\lambda$ be the partition obtained by removing this row, we have
\begin{align*}
M_\lambda^+(q,t^{-1})&=M_{\lambda''}^+(q,t^{-1})\left\langle\lambda\left|(\tau_\upsilon^+\circ\varpi)\left(h_{n^*}\left[(1-q\sigma^{-1})^{-1}(-X^{(p^*)}) \right]\right)\right|\lambda''\right\rangle\\
M_\lambda^-(q,t^{-1})&=M_{\lambda''}^-(q,t^{-1})\left\langle\lambda\left|(\tau_\upsilon^-\circ\varpi^{-1})\left(h_{n^*}\left[(1-q\sigma^{-1})^{-1}X^{(p^*)} \right]\right)\right|\lambda''\right\rangle
\end{align*}
\end{lem}

\begin{proof}
We will prove in detail the case of $N^-_\lambda(q,t^{-1})$.
First note that we have $e_{\quot(\lambda)}=e_n[X^{(p)}]e_{\quot(\lambda')}$.
Using (\ref{EHtoPQ}) for $\lambda'$, we have
\begin{align}
\nonumber
&(\tau_\upsilon^-\circ\varpi^{-1})\left( e_{\quot(\lambda)}\left[ (1-t^{-1}\sigma^{-1})^{-1}X^{\bullet} \right] \right)|\core(\lambda)\rangle\\
\label{EnLambda}
&=(\tau_\upsilon^-\circ\varpi^{-1}) \left( e_{n}\left[ (1-t^{-1}\sigma^{-1})^{-1}X^{(p)} \right] \right)
\left\{ N^-_{\lambda'}(q,t)|\lambda'\rangle+\sum_{\mu'<_\ell\lambda'}\gamma_{\lambda'\mu'}(q,t)N^-_{\mu'}(q,t)|\mu'\rangle \right\}
\end{align}
By Proposition \ref{ShufflePres}, we can study (\ref{EnLambda}) using the shuffle element $E_{p,n}$.
Proposition \ref{Adjacency} states that to a Fock basis element, the action of $E_{p,n}$ adds $n$ boxes of each color such that no $p$- and $(p+1)$-colored boxes are horizontally adjacent.
The following technical result is part of Lemma 5.2 of \cite{WreathEigen}:
\begin{lem*}
Suppose that $\lambda'$ is obtained from $\lambda$ by removing the last column of $\quot(\lambda)$ in accordance with the left-to-right order of columns.
Let that column be in the $p$-th component with length $n$ and let $\mu'\le_\ell\lambda'$. 
If we obtain $\mu$ by adding $n$ $i$-nodes for each $i\in\ZZ/\ell\ZZ$ to $\mu'$ in a manner where no added $p$- and $(p+1)$-nodes are horizontally adjacent, then $\mu=\lambda$ only if $\mu'=\lambda'$.
\end{lem*}
\noindent Thus, the only contribution to $|\lambda\rangle$ in (\ref{EnLambda}) comes from the summand $|\lambda'\rangle$.
The desired equation follows.
For $M^\pm_\lambda(q,t^{-1})$, we use Lemma 5.3 of \cite{WreathEigen}.
\end{proof}

\subsection{Conclusion}\label{Setup}
Combining Lemma \ref{Induct} with Propositions \ref{ShufflePres} and \ref{PosEnHn}, the computation of $N_\lambda^{\pm}(q,t^{-1})$ reduces to computing
\begin{align*}
&\frac{(-1)^n(qt)^n(\qqq-\qqq^{-1})^{n\ell}}{\prod_{r=1}^n(1-(qt)^r)}
\langle\lambda|(\tau^+_\upsilon\circ\Psi_+)(H_{p,n})|\lambda'\rangle\\
\hbox{and }&\frac{(1-qt)^{n\ell}}{\prod_{r=1}^n(1-(qt)^{-r})}
\langle\lambda|(\tau^-_\upsilon\circ\Psi_-)(E_{p,n})|\lambda'\rangle
\end{align*}
Here, $\lambda'\subset\lambda$ is obtained by removing the final column in the left-to-right order of columns of $\quot(\lambda)$ and said column is length $n$ in coordinate $p$.
Similarly, the computation of $M_\lambda^\pm(q,t^{-1})$ reduces to computing
\begin{align*}
&\frac{(-1)^n(\qqq-\qqq^{-1})^{n\ell}}{\prod_{r=1}^n(1- (qt)^r)}
\langle\lambda|(\tau^+_\upsilon\circ\Psi_+)(E_{p^*,n^*})|\lambda''\rangle\\
\hbox{and }&\frac{(qt)^{-n}(1-qt)^{n\ell}}{\prod_{r=1}^n(1-(qt)^{-r})}
\langle\lambda|(\tau^-_\upsilon\circ\Psi_-)(H_{p^*,n^*})|\lambda''\rangle
\end{align*}
where $\lambda''\subset\lambda$ is obtained by removing the final row in the right-to-left order of rows of $\quot(\lambda)$ and said row is length $n^*$ in coordinate $p^*$.
These formulas are nice in that they are inductive, but computing them using Proposition \ref{ShuffleFock} is still a formidable task due to the symmetrization in the formulas for $H_{p,n}$ and $E_{p,n}$.
Moreover, one needs to compute these normalization factors before using Proposition \ref{ShuffleFock} to compute Pieri rules.

We end by computing an extended example.
Let $\ell=3$ and consider $\lambda=(4,3,1)$, as depicted in Figure \ref{FinalExFig}.
\begin{figure}[h]
\centering
\begin{tikzpicture}
\draw (0,0)--(4,0)--(4,1)--(3,1)--(3,2)--(1,2)--(1,3)--(0,3)--(0,0);;
\draw (1,0)--(1,2);;
\draw (2,0)--(2,2);;
\draw (3,0)--(3,1);;
\draw (0,1)--(3,1);;
\draw (0,2)--(1,2);;
\draw (.5,.5) node {$1$};;
\draw (1.5,.5) node {$q$};;
\draw (2.5,.5) node {$q^2$};;
\draw (3.5,.5) node {$q^3$};;
\draw (.5,1.5) node {$t$};;
\draw (1.5,1.5) node {$qt$};;
\draw (2.5,1.5) node {$q^2t$};;
\draw (.5,2.5) node {$t^2$};;
\end{tikzpicture}
\caption{The partition $\lambda=(4,3,1)$ with the characters of its boxes.}
\label{FinalExFig}
\end{figure}
We then have $\core(\lambda)=(2)$ and $\quot(\lambda)=\left( \varnothing,\varnothing, (2) \right)$.
For $N_\lambda^-(q,t^{-1})$, we will need two applications of $(\tau_\upsilon^-\circ\Psi_-)(E_{2,1})$, whereas $M_\lambda^-(q,t^{-1})$ will only need one application of $(\tau_\upsilon^-\circ\Psi_-)(H_{2,2})$.
In the first application of $E_{2,1}$, we will add the boxes with characters $t, t^2, qt$ to yield $\lambda'=(2,2,1)$.
Altogether, we have
\[
N^-_{\lambda'}(q,t^{-1})=\frac{(1-qt)^3}{(1-(qt)^{-1})}\langle\lambda'|(\tau_\upsilon^-\circ\Psi_-)(E_{2,1})|\core(\lambda)\rangle=-(q^{-1}-qt^{-1})qt^3
\]
It is worth noting that
\[
M^-_{\lambda'}(q,t^{-1})=\frac{(1-qt)^3}{qt(1-(qt)^{-1})}\langle\lambda'|(\tau_\upsilon^-\circ\Psi_-)(H_{2,1})|\core(\lambda)\rangle=-(q^{-1}-t^{-2})qt^3
\]
and thus
\[
\langle P^*_{\tensor[^t]{\lambda}{}'},P_{\lambda'}\rangle_{q,t}=\frac{N^-_{\lambda'}(q,t)}{M^-_{\lambda'}(q,t)}=\frac{1-q^2t}{1-qt^2}
\]
which confirms Conjecture \ref{NormConj} for $\lambda'=(2,2,1)$.

Moving onwards, we perform the second application of $E_{2,1}$, which appends the boxes with characters $q^2, q^2t, q^3$:
\begin{align*}
N_\lambda^-(q,t^{-1})&= N_{\lambda'}^-(q,t^{-1})\frac{(1-qt)^3}{(1-(qt)^{-1})}\langle\lambda|(\tau_\upsilon^-\circ\Psi_-)(E_{2,1})|\lambda'\rangle\\
&= -\ddd^{-5}t^4(1-q^4t^{-2})(1-q^2t^{-1})
\end{align*}
Finally, we compute $M_\lambda^-(q,t^{-1})$.
Due to the mixing terms in the formula (\ref{Hpn}) for $H_{2,2}$, all but two summands of the symmetrization will vanish upon evaluation\footnote{More summands can appear for general partitions, even for degree two Pieri operators.
For degree $n$, there will be at least $n!$ summands.}.
We obtain:
\begin{align*}
M_\lambda^-(q,t^{-1})&= \frac{(1-qt)^6}{q^2t^2\left(1-(qt)^{-1}\right)\left( 1-(qt)^{-2} \right)}\langle\lambda|(\tau_\upsilon\circ\Psi_-)(H_{2,2})|\core(\lambda)\rangle\\
&= \qqq^6(-\ddd)^{-7}t^2\frac{(1-q^3)(1-qt^{-2})^2(1-q^3t^{-3})(1-qt)}{(1-q^2t^2)(1-q^2t^{-1})}
\left( \frac{1}{1-q^3} 
+
\frac{1}{qt(1-qt^{-2})}
\right)\\
&= \qqq^6(-\ddd)^{-7}t^2\frac{(1-q^3)(1-qt^{-2})^2(1-q^3t^{-3})(1-qt)}{(1-q^2t^2)(1-q^2t^{-1})}
\cdot\frac{1-q^3+qt-q^2t^{-1}}{qt(1-q^3)(1-qt^{-2})}\\
&= \qqq^4(-\ddd)^{-7}t^2\frac{(1-qt^{-2})(1-q^3t^{-3})(1-qt)}{(1-q^2t^2)(1-q^2t^{-1})}
\cdot(1+qt)(1-q^2t^{-1})\\
&= -\ddd^{-5}t^4(1-qt^{-2})(1-q^3t^{-3})
\end{align*}
Therefore, we have
\[
\langle P_{\tensor[^t]{\lambda}{}}^*,P_{\lambda}\rangle_{q,t}=\frac{N_\lambda^-(q,t)}{M_\lambda^-(q,t)}
=\frac{(1-q^4t^2)(1-q^2t)}{(1-q^3t^3)(1-qt^2)}
\]
which again confirms Conjecture \ref{NormConj}.

%

\bibliographystyle{alpha}
\bibliography{Wreath}

\end{document}